\documentclass[12pt,a4paper,oneside]{amsart}

\usepackage{typearea}
\usepackage[draft]{fixme}

\usepackage[utf8]{inputenc}
\usepackage[T1]{fontenc}
\usepackage{amsthm}
\usepackage{xspace}
\usepackage{amssymb,amsmath,stmaryrd}
\usepackage[english]{babel}
\usepackage[shortlabels]{enumitem}
\usepackage[all]{xy}
\usepackage{caption}
\usepackage{subcaption}
\usepackage{wrapfig}
\usepackage{xcolor}
\usepackage{xr-hyper}
\usepackage{hyperref}
\theoremstyle{plain}
\newtheorem{theorem}{Theorem}[section]
\newtheorem{remark}[theorem]{Remark}
\newtheorem{lemma}[theorem]{Lemma}
\newtheorem{corollary}[theorem]{Corollary}
\newtheorem{proposition}[theorem]{Proposition}

\theoremstyle{definition}

\newtheorem{definition}[theorem]{Definition}
\newtheorem{notation}[theorem]{Notation}

\numberwithin{equation}{section}
\newcommand{\fct}[3]{\ensuremath{#1\colon #2\rightarrow #3}\xspace}
\newcommand{\fctw}[3]{$#1$ from $#2$ to $#3$\xspace}
\newcommand{\ca}{\mbox{$C\sp*$-}al\-ge\-bra\xspace}
\newcommand{\cas}{\mbox{$C\sp*$-}al\-ge\-bras\xspace}
\newcommand{\starhomo}{\mbox{$\sp*$-}ho\-mo\-morphism\xspace}

\newcommand{\stariso}{\mbox{$\sp*$-}iso\-morphism\xspace}

\newcommand{\Z}{\ensuremath{\mathbb{Z}}\xspace}
\newcommand{\N}{\ensuremath{\mathbb{N}}\xspace}

\newcommand{\K}{\ensuremath{\mathbb{K}}\xspace}
\newcommand{\ie}{\emph{i.e.}\xspace}
\newcommand{\cf}{\emph{cf.}\xspace}
\newcommand{\eg}{\emph{e.g.}\xspace}
\newcommand{\id}{\ensuremath{\operatorname{id}}\xspace}

\newcommand{\cok}{\operatorname{cok}}

\newcommand{\Ocal}{\ensuremath{\mathcal{O}}\xspace}

\newcommand{\FKR}{\ensuremath{\operatorname{FK}_\mathcal{R}}\xspace}
\newcommand{\FKRs}{\ensuremath{\operatorname{FK}_\mathcal{R}^s}\xspace}

\newcommand{\FK}{\ensuremath{\operatorname{FK}}\xspace}
\newcommand{\A}{\ensuremath{\mathfrak{A}}\xspace}
\newcommand{\B}{\ensuremath{\mathfrak{B}}\xspace}
\newcommand{\Asf}{\mathsf{A}}
\newcommand{\Bsf}{\mathsf{B}}

\newcommand{\Prim}{\operatorname{Prim}}

\newcommand{\calP}{\ensuremath{\mathcal{P}}\xspace}
\newcommand{\GLZ}[1][n]{\ensuremath{\operatorname{GL}(#1,\Z)}\xspace}
\newcommand{\GL}{\ensuremath{\operatorname{GL}}\xspace}

\newcommand{\GLPZ}[1][\mathbf{n}]{\ensuremath{\operatorname{GL}_\calP(#1,\Z)}\xspace}
\newcommand{\GLP}{\ensuremath{\operatorname{GL}_\calP}\xspace}
\newcommand{\SLPZ}[1][\mathbf{n}]{\ensuremath{\operatorname{SL}_\calP(#1,\Z)}\xspace}
\newcommand{\SLP}{\ensuremath{\operatorname{SL}_\calP}\xspace}

\newcommand{\MPZ}[1][\mathbf{m}\times\mathbf{n}]{\ensuremath{\mathfrak{M}_\calP(#1,\Z)}\xspace}
\newcommand{\MPZc}[1][\mathbf{m}\times\mathbf{n}]{\ensuremath{\mathfrak{M}^\circ_\calP(#1,\Z)}\xspace}
\newcommand{\MPZcc}[1][\mathbf{m}\times\mathbf{n}]{\ensuremath{\mathfrak{M}^{\circ\circ}_\calP(#1,\Z)}\xspace}
\newcommand{\MPZccc}[1][\mathbf{m}\times\mathbf{n}]{\ensuremath{\mathfrak{M}^{\circ\circ\circ}_\calP(#1,\Z)}\xspace}

\newcommand{\MPplusZ}[1][\mathbf{m}\times\mathbf{n}]{\ensuremath{\mathfrak{M}^+_\calP(#1,\Z)}\xspace}
\newcommand{\ftn}[3]{ #1 \colon #2 \rightarrow #3 }
\newcommand{\setof}[2]{\left\{ #1 \;\middle|\; #2 \right\}}

\newcommand{\GLPEe}{\GLP-equivalence\xspace}
\newcommand{\GLPEes}{\GLP-equivalences\xspace}
\newcommand{\SLPEe}{\SLP-equivalence\xspace}

\newcommand{\Meq}{\ensuremath{\sim_{M\negthinspace E}}\xspace}

\newenvironment{smallpmatrix}{\left(\begin{smallmatrix}}{\end{smallmatrix}\right)}

\newcommand{\first}[1]{}
\newcommand{\second}[1]{}
\newcommand{\third}[1]{}

\makeatletter
\def\citefirst{\@ifnextchar[{\@with}{\@without}}
\def\@with[#1]{\cite[#1]{arXiv:1604.05439v1}\xspace}
\def\@without{\cite{arXiv:1604.05439v1}\xspace}
\makeatother

\title
{Strong classification of purely infinite Cuntz-Krieger algebras}
\date{\today}
\author{Toke Meier Carlsen}
\address{Department of Science and Technology, University of the Faroe Islands, N\'{o}at\'{u}n~3, FO-100 T\'{o}rshavn, the Faroe Islands}
\email{tokemc@setur.fo}
\author{Gunnar Restorff}
\address{Department of Science and Technology, University of the Faroe Islands, N\'{o}at\'{u}n~3, FO-100 T\'{o}rshavn, the Faroe Islands}
\email{gunnarr@setur.fo}
\author{Efren Ruiz}
\address{Department of Mathematics, University of Hawaii, Hilo, 200 W.~Kawili St., Hilo, Hawaii, 96720-4091 USA}
\email{ruize@hawaii.edu}
\keywords{Cuntz-Krieger algebras, Graph $C^*$-algebras, $K$-theory, Flow equivalence}
\subjclass[2010]{46L35, 46L80 (46L55, 37B10)}

\begin{document}

\begin{abstract}
In 2006, Restorff completed the classification of all Cuntz-Krieger algebras with finitely many ideals (\ie, those that are purely infinite) up to stable isomorphism.  He left open the questions concerning strong classification up to stable isomorphism and unital classification.  In this paper, we address both questions.  We show that any isomorphism between the reduced filtered $K$-theory of two Cuntz-Krieger algebras with finitely many ideals lifts to a \stariso between the stabilized Cuntz-Krieger algebras.  As a result, we also obtain strong unital classification.
\end{abstract}

\maketitle


\section{Introduction}

Historically there has been a lot of connection between classification of shifts of finite type, classification of Cuntz-Krieger algebras and classification of more general \cas. 
Franks made a successful classification of irreducible shifts of finite type up to flow equivalence (\cite{MR758893}), which R\o{}rdam used to classify simple, purely infinite Cuntz-Krieger algebras (\cite{MR1340839}) using a trick of Cuntz. 
This was shortly after generalized by Kirchberg-Phillips to a \emph{strong} classification of simple, nuclear, purely infinite \cas in the UCT-class, \ie, allowing for lifting of isomorphisms on the invariant level to the algebra level (\cf\ \cite{MR1745197,MR1796912}). 
Huang classified shifts of finite type with finite Bowen-Franks groups up to flow equivalence and used this to classify Cuntz-Krieger algebras with finite $K_0$-groups (\cite{MR1329907}).
Huang also classified two component shifts of finite type up to flow equivalence and used this to classify Cuntz-Krieger algebras with exactly one ideal (\cite{MR1304139,MR1301504}). This was generalized by R\o{}rdam to essential extensions of stable, simple, nuclear, purely infinite \cas in the UCT-class (\cite{MR1446202}). 
Eilers-Restorff and Restorff-Ruiz (\cite{MR2265044,MR2379290}) generalized this to cover classification for 
all essential extensions of simple, nuclear, purely infinite \cas in the UCT-class using results of Kirchberg and Bonkat (\cite{MR1796912,bonkat:phd}) --- allowing for strong classification in the unital and the stable cases. 
The work in \cite{MR2265044,MR2379290,arXiv:1301.7695v1} shows that for properly infinite \cas (or for \cas with stable weak cancellation), the key to solving the unital classification problem is to get a strong classification for the stable case.

Finally, Boyle and Huang classified general shifts of finite type up to flow equivalence (\cite{MR1907894,MR1990568}). Using this, Restorff classified all purely infinite Cuntz-Krieger algebras up to stable isomorphism --- with the problems of strong classification and unital classification left as open questions in the addendum (\cite[Questions~1 and 3]{MR2270572}). These questions will be answered in this paper.

Restorff, Meyer-Nest, Bentmann-K\"ohler have shown that stable, nuclear, purely infinite \cas in the UCT-class with finite ideal lattice are classified by their filtered $K$-theory if and only if their primitive ideal spaces are so-called accordion spaces --- and when there is classification, it is a strong classification (\cite{restorff:phd,MR2953205,arXiv:1101.5702v3}). In contrast, all finite ideal spaces appear as the ideal space of Cuntz-Krieger algebras. Arklint-Restorff-Ruiz (\cite{arXiv:1405.0672v1}) showed that there are purely infinite Cuntz-Krieger algebras with four primitive ideals for which there are automorphisms of the filtered $K$-theory that cannot be lifted to the stabilized algebra, while the work of Arklint-Restorff-Ruiz and Arklint-Bentmann-Katsura shows that all Cuntz-Krieger algebras with at most four primitive ideals are strongly classified by the \emph{reduced} filtered $K$-theory --- both up to stable isomorphism and up to unital isomorphism (\cite{arXiv:1405.0672v1,MR3349327}). 
This suggests that the correct invariant to strongly classify Cuntz-Krieger algebras is the reduced filtered $K$-theory. 

Move equivalence and Cuntz move equivalence of graphs, Cuntz-Krieger algebras and graph \cas and their interplay have recently also been important for developing results about classification of certain nonsimple \cas that are of mixed kind (\eg\ \cite{arXiv:1604.05439v1,Eilers-Restorff-Ruiz-Sorensen-2,arXiv:1602.03709v2,MR3056712,MR3142033,MR2562779,MR2563693}) --- since we will only focus on strong classification of purely infinite Cuntz-Krieger algebras, we will not examine this further here. 

In this paper, we will strongly classify all purely infinite Cuntz-Krieger algebras up to stable isomorphism and up to unital isomorphism using the reduced filtered $K$-theory as invariant, together with a way of representing the class of the unit for the unital classification (thus providing an answer to Question 1 and a positive answer to Question~3 in the addendum in \cite{MR2270572}). 
We will do so drawing heavily on results from \cite{MR1907894,MR1990568,MR2270572,arXiv:1301.7695v1,arXiv:1604.05439v1,arXiv:1602.03709v2} and using graph \cas instead of Cuntz-Krieger algebras will be of key importance for the proof. 
Along the way, we will therefore show some results in slightly greater generality than strictly needed for our main theorems --- these results will be of great importance in a forthcoming paper about the geometric classification of general unital graph \cas (\cite{Eilers-Restorff-Ruiz-Sorensen-2}).

We also note that Arklint-Bentmann-Katsura have a very nice range result characterizing the range of the reduced invariant for all purely infinite Cuntz-Krieger algebras --- both with and without the class of the unit (\cite{MR3177344}). 

Before we state the main theorems in Section~\ref{sec:main}, we introduce in Section~\ref{sec:prelim} some notation and definitions needed to state the main results. In Section~\ref{sec:main}, we state the the main results of this paper, and we then explain the strategy of the proofs of these results, and introduce some notation and definitions we use in the proofs. In Section~\ref{sec:lifting}, we strengthen some results of \cite{MR1990568}, and in Section~\ref{sec:edge}, we describe how edge expanding a graph gives rise to stable isomorphism of graph algebras and \SLPEe of corresponding matrices derived from the graphs. In Section~\ref{sec:toke}, we examine how row and column addition of matrices give rise to stable isomorphism of graph \cas, and describe what maps on the reduced filtered $K$-theory these stable isomorphisms induce. Finally, in Section~\ref{sec:cuntzsplice} we look at how Cuntz splicing a graph gives rise to stable isomorphism of graph \cas and describe what maps on the reduced filtered $K$-theory such a stable isomorphism induces, before we give the proofs of the main theorems.

\section{Preliminaries}
\label{sec:prelim}

We now introduce and recall some notation and definitions needed for stating the main results of this paper. 
Let $A$ be an $n\times n$ matrix with entries in $\N_0$, the nonnegative integers, and let $A(i,j)$ denote the $(i,j)$'th entry of $A$. 
We say that $A$ is \emph{nondegenerate} if all its rows and columns are nonzero.

For every nondegenerate $n\times n$ matrix $A$ with entries from $\{0,1\}$, the \emph{Cuntz-Krieger algebra}, $\Ocal_A$, associated with $A$ is the universal (unital) \ca generated by $n$ partial isometries $s_1,\ldots,s_n$ satisfying the relations
\begin{align*}
\mathbf{1}&=s_1s_1^*+\cdots+s_ns_n^*, \\
s_i^*s_i&=\sum_{j=1}^n A(i,j)s_js_j^*, \qquad\text{for all }i=1,\ldots,n.
\end{align*}

It is a well-known fact that Cuntz-Krieger algebras are unital, separable, nuclear \cas in the UCT-class. 
In \cite{MR608527}, Cuntz defined Condition~(II) for a nondegenerate matrix with entries from $\{0,1\}$. 
A nondegenerate matrix $A$ with entries from $\{0,1\}$ satisfies Condition~(II) if and only if the Cuntz-Krieger algebra $\Ocal_A$ is purely infinite if and only if the Cuntz-Krieger algebra $\Ocal_A$ has finitely many ideals if and only if the Cuntz-Krieger algebra $\Ocal_A$ has real rank zero. See \cite{MR2270572} and the references therein. 

For a \ca \A, we let $\Prim\A$ denote the primitive ideal space equipped with the usual hull-kernel topology. 
It is a well-known fact that for every open subset $O$ of $\Prim\A$ the set $\A(O)=\cap((\Prim\A)\setminus O)$ is a (two-sided, closed) ideal of \A, and that this is a lattice isomorphism from the lattice of open subsets of $\Prim\A$ to the lattice of ideals of \A. 

If \A has finitely many ideals, then for every subset $S$ of $\Prim\A$ there exists a smallest open subset $\widetilde S$ of $\Prim\A$ that contains $S$. For convenience, we will use the notation $\widetilde{\mathfrak{p}}$ for $\widetilde{\{\mathfrak{p}\}}$, and note that $\widetilde{\partial}\mathfrak{p} :=\widetilde{\mathfrak{p}}\setminus\{\mathfrak{p}\}$ is also open. The following definition is a reformulation of \cite[Definition~4.1]{MR2270572}, the definition can be generalized to \cas over finite $T_0$-spaces (see \eg\  \cite{MR2949216,MR3349327,arXiv:1604.05439v1}). 

\begin{definition}
Let \A be a \ca with finitely many ideals. 
Let 
\begin{align*}
I_0(\A)&=\setof{\widetilde{\mathfrak{p}}}{\mathfrak{p}\in\Prim\A}
\cup\setof{\widetilde{\partial}\mathfrak{p}}{\mathfrak{p}\in\Prim\A},\\
C(\A)&=\setof{(\mathfrak{p},\mathfrak{q})\in\Prim\A\times\Prim\A}{\widetilde{\mathfrak{p}}\subsetneq\widetilde{\partial}\mathfrak{q}\text{ and }\not\exists \mathfrak{p}'\in \Prim\A\text{ such that }\widetilde{\mathfrak{p}}\subsetneq\widetilde{\mathfrak{p}'}\subseteq \widetilde{\partial}\mathfrak{q}}.
\end{align*}
The \emph{reduced filtered $K$-theory} of \A, $\FK_{\textrm{red}}(\A)$, consists of $\Prim\A$ and the families 
$$(K_0(\A(O)))_{O\in I_0(\A)},\qquad (K_0(\A(\widetilde{\mathfrak{p}})/\A(\widetilde{\partial}\mathfrak{p})))_{\mathfrak{p}\in \Prim\A},\qquad
(K_1(\A(\widetilde{\mathfrak{p}})/\A(\widetilde{\partial}\mathfrak{p})))_{\mathfrak{p}\in \Prim\A},$$
together with the sequences 
$$\xymatrix@C=15pt@R=12pt{
K_1(\A(\widetilde{\mathfrak{p}})/\A(\widetilde{\partial}\mathfrak{p}))
\ar[r]&
K_0(\A(\widetilde{\partial}\mathfrak{p}))\ar[r]&
K_0(\A(\widetilde{\mathfrak{p}}))\ar[r] & 
K_0(\A(\widetilde{\mathfrak{p}})/\A(\widetilde{\partial}\mathfrak{p}))}$$
for every $\mathfrak{p}\in\Prim\A$ and 
$$\xymatrix@C=15pt{& K_0(\widetilde{\mathfrak{p}})\ar[r] & K_0(\widetilde{\partial}\mathfrak{q})}$$
for every pair $(\mathfrak{p},\mathfrak{q})\in C(\A)$
originating from the cyclic six term exact sequences induced by the extensions $\A(\widetilde{\partial}\mathfrak{p})\hookrightarrow 
\A(\widetilde{\mathfrak{p}})\twoheadrightarrow
\A(\widetilde{\mathfrak{p}})/\A(\widetilde{\partial}\mathfrak{p})$
and 
$\A(\widetilde{\mathfrak{p}})\hookrightarrow 
\A(\widetilde{\partial}\mathfrak{q})\twoheadrightarrow 
\A(\widetilde{\partial}\mathfrak{q})/\A(\widetilde{\mathfrak{p}})$, respectively.

Let $\B$ be another \ca with finitely many ideals. 
By a \emph{reduced filtered $K$-theory isomorphism} \fctw{(\kappa,\rho)}{\FK_{\textrm{red}}(\A)}{\FK_{\textrm{red}}(\B)} we understand isomorphisms
\begin{align*}
\rho\colon \Prim\A&\rightarrow\Prim\B, &&\text{homeomorphism},\\
\kappa^{0,i}_O\colon K_0(\A(O))&\rightarrow K_0(\B(\rho(O))), &&\text{for }O\in I_0(\A),\\
\kappa^{0,q}_\mathfrak{p}\colon K_0(\A(\widetilde{\mathfrak{p}})/\A(\widetilde{\partial}\mathfrak{p}))&\rightarrow K_0(\B(\widetilde{\rho(\mathfrak{p})})/\B(\widetilde{\partial}\rho(\mathfrak{p}))), &&\text{for }\mathfrak{p}\in \Prim\A,\\
\kappa^{1,q}_\mathfrak{p}\colon K_1(\A(\widetilde{\mathfrak{p}})/\A(\widetilde{\partial}\mathfrak{p}))&\rightarrow K_1(\B(\widetilde{\rho(\mathfrak{p})})/\B(\widetilde{\partial}\rho(\mathfrak{p}))), &&\text{for }\mathfrak{p}\in \Prim\A,
\end{align*}
such that all the ladders coming from sequences in $\FK_{\textrm{red}}(\A)$ and $\FK_{\textrm{red}}(\B)$ commute. 

It is clear that any \stariso $\Psi$ from $\A$ to $\B$ or from $\A\otimes\K$ to $\B\otimes\K$ induces in a canonical way an isomorphism from $\FK_{\mathrm{red}}(\A)$ to $\FK_{\mathrm{red}}(\B)$. We denote these by $\FK_{\mathrm{red}}(\Psi)$ and $\FK_{\mathrm{red}}^s(\Psi)$, respectively. 
\end{definition}

There is a nice description of the ideal lattice and the $K$-theory of all subquotients as well as of the occurring cyclic six-term exact sequences in terms of the underlying matrix (see \cite{MR2270572} and the references therein).

\begin{definition}\label{def:red-filtered-Ktheory-class-unit}
Let $\A$ and $\B$ be \cas with finitely many ideals and having real rank zero.  By \cite[Lemma~8.3]{MR3349327}, every isomorphism \fctw{(\kappa,\rho)}{\FK_{\textrm{red}}(\A)}{\FK_{\textrm{red}}(\B)} induces a unique isomorphism \fct{\kappa_0}{ K_0 (\A) }{ K_0 ( \B ) }.  If $\A$ and $\B$ are unital, we say the reduced filtered $K$-theory isomorphism \fctw{(\kappa,\rho)}{\FK_{\textrm{red}}(\A)}{\FK_{\textrm{red}}(\B)} \emph{preserves the class of the unit} if $\kappa_0 ( [ 1_\A]_0 ) = [ 1_\B]_0$ in $K_0 (\B)$.  We write $( \FK_{\textrm{red}} (\A) , [1_\A]_0 ) \cong ( \FK_{\textrm{red}} (\B) , [1_\B]_0 )$ whenever such an isomorphism exists.
\end{definition}

\section{Main results}

\label{sec:main}

We now state the two main theorems of the paper. The proofs of Theorems~\ref{thm:strong-stable} and~\ref{thm:strong-unital} are given at the end of Section~\ref{sec:cuntzsplice}.

\begin{theorem}\label{thm:strong-stable}
Let $A$ and $A'$ be nondegenerate matrices with entries from $\{0,1\}$ satisfying Condition~(II). 
For every isomorphism \fctw{(\kappa,\rho)}{\FK_{\mathrm{red}}(\Ocal_A)}{\FK_{\mathrm{red}}(\Ocal_{A'})} there exists a \stariso \fctw{\Psi}{\Ocal_A\otimes\K}{\Ocal_{A'}\otimes\K} such that $\FK_{\mathrm{red}}^s(\Psi)=(\kappa,\rho)$. 
\end{theorem}

\begin{theorem}\label{thm:strong-unital}
Let $A$ and $A'$ be nondegenerate matrices with entries from $\{0,1\}$ satisfying Condition~(II). 
For every isomorphism \fctw{(\kappa,\rho)}{\FK_{\mathrm{red}}(\Ocal_A)}{\FK_{\mathrm{red}}(\Ocal_{A'})} that preserves the class of the unit, there exists a \stariso \fctw{\Psi}{\Ocal_A}{\Ocal_{A'}} such that $\FK_{\mathrm{red}}(\Psi)=(\kappa,\rho)$. 
\end{theorem}

We immediately get the following corollary, where the first half of course is already proved earlier in \cite{MR2270572}.
\begin{corollary}
Let $A$ and $A'$ be nondegenerate matrices with entries from $\{0,1\}$ satisfying Condition~(II). 

Then $\Ocal_A\otimes\K\cong\Ocal_{A'}\otimes\K$ if and only if there exists an isomorphism from $\FK_{\mathrm{red}}(\Ocal_A)$ to $\FK_{\mathrm{red}}(\Ocal_{A'})$, and, moreover, $\Ocal_A\cong\Ocal_{A'}$ if and only if there exists an isomorphism from $\FK_{\mathrm{red}}(\Ocal_A)$ to $\FK_{\mathrm{red}}(\Ocal_{A'})$ that preserves the class of the unit.
\end{corollary}

\subsection{Proofs of Theorems~\ref{thm:strong-stable} and~\ref{thm:strong-unital}.}
The proofs of Theorem~\ref{thm:strong-stable} and Theorem~\ref{thm:strong-unital} will be carried out in the remaining sections of the paper. 
The idea of the proof of Theorem~\ref{thm:strong-stable} is relatively simple. The isomorphism $(\kappa,\rho)$ is induced by a \GLPEe between $A-I$ and $A'-I$. By doing Cuntz-splices on $A'$ (see Section~\ref{sec:cuntzsplice}), we can change this \GLPEe into an \SLPEe, and by the results of Boyle in \cite{MR1907894} such an \SLPEe can be decomposed into a composition of \SLPEe{s} given by elementary matrices. Finally, for each \SLPEe given by an elementary matrix we can construct a stable isomorphism that induces the same map on the filtered $K$-theory as the \SLPEe (see Section~\ref{sec:toke}). Since \GLPEe{s} might take us out of the class of $\{0,1\}$ matrices, we work in the rest of the paper with graph \cas rather than Cuntz-Krieger algebras.
Also for the benefit of later applications, we prove some results in slightly greater generality than we need --- allowing for singular vertices sometimes. 
Theorem~\ref{thm:strong-unital} will follow from Theorem~\ref{thm:strong-stable} and \cite[Theorem~3.3]{arXiv:1301.7695v1}.

\subsection{Definitions and notation}

We adopt verbatim the definitions and the notation from \cite{arXiv:1604.05439v1} --- in particular Sections \ref{I-genprel}, \ref{I-gipis} and \ref{I-sec:notation-for-proof}. Since all notation needed is defined and explained in \cite{arXiv:1604.05439v1}, we will only recall some of it here, but refer the reader to that paper for the rest.

\begin{definition}
For a countable graph $E = (E^0 , E^1 , r, s)$ (\ie, $E^0$ and $E^1$ are countable sets), we define its \emph{adjacency matrix} $\Asf_E$ as an $E^0\times E^0$ matrix with the $(u,v)$'th entry being
$$\left\vert\setof{e\in E^1}{s(e)=u, r(e)=v}\right\vert.$$
Since $E$ is countable, $\Asf_E$ will be a finite matrix or a countably infinite matrix, and it will have entries from $\N_0\sqcup\{\infty\}$.
\end{definition}

It will be convenient for us to alter the adjacency matrix of a graph in two very specific ways, removing singular rows and subtracting the identity, so we introduce notation for this. 

\begin{notation}
Let $E$ be a graph and $\Asf_E$ its adjacency matrix.  The matrix $\Asf_{E}^\bullet$ will denote the matrix obtained from $\Asf_{E}$ by removing all rows corresponding to singular vertices of $E$.  We define $\Bsf_E := \Asf_{E} - I$ and let $\Bsf_{E}^\bullet$ be $\Bsf_E$ with the rows corresponding to singular vertices of $E$ removed. 

We will mainly be working with graphs $E$ such that $\Bsf_E^\bullet \in \MPZ$, $\Bsf_E \in \MPZc$, $\Bsf_E \in \MPZcc$, or $\Bsf_E \in \MPZccc$ (see \citefirst[Definition~\ref{I-def:blockmatrices} and Definition~\ref{I-def:circ}]).
If $E$ is a finite graph with no sinks, then $\Bsf_E^\bullet=\Bsf_E$.
\end{notation}

\begin{notation}
For an arbitrary unital graph \ca $C^*(E)$, we will consider a slightly different definition of reduced filtered $K$-theory.  We refer the reader to \citefirst[Sections~\ref{I-sec:reducedKtheory} and~4.3], for the definition of $\FKR(X; C^*(E))$ and $\FKR(\calP;C^*(E))$, where $X$ is a finite $T_0$-space and $\calP$ is a partially ordered set.
When $E$ is a finite graph with no sinks and no sources that satisfies Condition~(K) (\ie, $C^*(E)$ is isomorphic to a purely infinite Cuntz-Krieger algebra), then the two definitions of reduced filtered $K$-theory are equivalent.
\end{notation}

\begin{remark}
Let $E$ and $F$ be graphs such that $\Bsf_E^\bullet$ and $\Bsf_F^\bullet$ are elements in \MPZ.  Let $(U,V)$ be a \GLPZ-equivalence or an \SLPZ-equivalence (see \citefirst[Definition~\ref{I-def:glpandslpeq}]) from $\Bsf_E^\bullet$ to $\Bsf_F^\bullet$.  As explained in \citefirst[Section~\ref{I-sec:notation-for-proof}], in the induced isomorphisms on the reduced filtered $K$-theory, $\ftn{ \FKR(U,V) }{ \FKR (\calP , C^*(E) )}{  \FKR (\calP , C^*(F) ) }$, $V^\mathsf{T}$ induces the map on $K_0$ and $(U^\mathsf{T})^{-1}$ induces the map on $K_1$.
Also, we let $\FKRs(U,V)$ denote the canonically induced isomorphism from $\FKR (\calP , C^*(E) \otimes \K )$ to $\FKR (\calP , C^*(F) \otimes \K )$.
\end{remark}

\subsection{Strategy of proof and structure of the paper}  The outline of the proof of Theorem~\ref{thm:strong-stable} is as follows.  We emulate the previous proofs that go from filtered $K$-theory data to stable isomorphism or flow equivalence, as in \cite{MR1990568, MR1907894, MR2270572}.
A key component of those proofs is manipulation of the matrix $\Bsf_E$, in particular that we can perform certain basic row and column operations without changing the stable isomorphism class or the flow equivalence class, depending on the context (\cf\ \cite[Propositions~\ref{CS-prop:columnAdd} and~\ref{CS-prop:rowAdd}]{arXiv:1602.03709v2}).
In Section~\ref{sec:toke}, we show how to describe the isomorphisms these matrix operations induce on the (stabilized) graph \cas and their reduced filtered $K$-theory. 
In \cite{arXiv:1602.03709v2}, it was proved that Cuntz splicing a graph gives a graph whose \ca is stably isomorphic to the \ca of the original graph (generalizing a result in \cite{MR2270572}). As we are proving a strong classification result, we will need to know more about what this stable isomorphism induces on the reduced filtered $K$-theory. We deal with this in Section~\ref{sec:cuntzsplice}.
We will also need to increase the size of our matrices using edge expansion and we keep track of the induced isomorphism on reduced filtered $K$-theory.  This is done in Section~\ref{sec:edge}.
Since the reduced filtered $K$-theory of $C^*(E)$ contains more groups than the $K$-web of $\Bsf_E$, we need to strengthen some of the results of \cite{MR1990568}. This is done in Section~\ref{sec:lifting}. 
Using these results, our proof of Theorem~\ref{thm:strong-stable} goes through 6 steps.

\begin{enumerate}[(1)]
\item[Step 1] First, the homeomorphism $\rho$ from $\Prim (\Ocal_A)$ to $\Prim (\Ocal_{A'})$ induces an action as a \ca over $\Prim (\Ocal_A)$ on $C^*(\Ocal_{A'})$.  Thus, $\kappa$ induces an isomorphism from $\FKR( \Prim ( \Ocal_A ) , \Ocal_A )$ to $\FKR( \Prim ( \Ocal_A ), \Ocal_{A'})$.

\item[Step 2] Secondly, we note that we might as well show the theorem for the graph \cas $C^*(E)$ and $C^*(E')$ instead of the Cuntz-Krieger algebras $\Ocal_A$ and $\Ocal_{A'}$, where $E$ and $E'$ are finite graphs satisfying Condition~(K) with no sinks and no sources, since every Cuntz-Krieger algebra is isomorphic to such a graph \ca. 
The reason is, that it is important to our proof techniques to be able to work with matrices that are not $\{0,1\}$-matrices, since we will do row and column operations as mentioned above. 

\item[Step 3] Now, we reduce the problem by finding graphs $E_1$ and $E_1'$ in a certain standard form such that we have equivariant isomorphisms from $C^*(E) \otimes\K$ to $C^*(E_{1}) \otimes\K$ and from $C^*(E') \otimes\K$ to $C^*(E_{1}') \otimes\K$. This standard form will ensure that the matrices $\Asf_{E_1}$ and $\Asf_{E_1'}$ have the same size and block structure, and that they satisfy certain additional technical conditions. This will be done in Lemmas~\ref{lem:stdform} and~\ref{lem:enlargen}. 
 
 \item[Step 4] In Section~\ref{sec:lifting} we show a result similar to \cite[Corollary~4.7]{MR1990568} that shows that an isomorphism $\FKR(\calP, C^*(E_{1}) ) \cong \FKR(\calP, C^*( E_{1}') )$ is induced by a \GLPEe from $\Bsf_{E_1}$ to $\Bsf_{E_1'}$. 
 
\item[Step 5] We then use the results in Section~\ref{sec:cuntzsplice} to find graphs $E_2, E_2'$ in a certain standard form such that we have equivariant isomorphisms from $C^*(E_{1}) \otimes\K$ to $C^*(E_{2}) \otimes\K$ and from $C^*(E_{1}') \otimes\K$ to $C^*(E_{2}') \otimes\K$ that induce isomorphisms on the reduced filtered $K$-theory that are induced by \GLPEes such that the composed \GLPEe from $\Bsf_{E_2}$ to $\Bsf_{E_2'}$ is in fact an \SLPEe. 
 
 \item[Step 6] The result will now follow from Proposition~\ref{prop:toke} and~\cite[Theorem~4.4]{MR1907894}.
 \end{enumerate}

\section{Lifting reduced filtered isomorphism to a \GLPEe}
\label{sec:lifting}

In \cite{MR1990568} there are many useful theorems concerning lifting $K$-web isomorphisms to \GLPEe{s} and \SLPEe{s}. Note that all $n_i$'s are assumed to be nonzero in \cite{MR1990568}. 
From \cite[Theorems~4.4 and~4.5]{MR1990568} we get the useful Corollary~4.7 in \cite{MR1990568}, that allows us to lift $K$-web isomorphisms to \GLPEe{s} provided that the $\gcd$ of each diagonal block is $1$. 
This corollary is easily generalized to also allow for diagonal blocks with $\gcd B\{i\}=\gcd B'\{i\}=0$ (\ie, $B\{i\}$ and $B'\{i\}$ are the zero matrices); this follows easily, since every automorphism of the cokernel of a (square) zero matrix is \GL-allowable (\cf\ \cite{MR1990568}). 

The reduced filtered $K$-theory includes more information in general than the $K$-webs (every $K_1$-group of every gauge simple subquotient, \cf\ the discussion in \citefirst[Section~\ref{I-sec:red-filtered-K-theory-K-web-GLP-and-SLP-equivalences}]). Therefore, we will need to strengthen the result \cite[Corollary~4.7]{MR1990568} even more. 
First, we generalize \cite[Theorem~4.4 (and Proposition~4.1)]{MR1990568} as follows (we here only consider the case $\mathcal{R}=\Z$). 

\begin{theorem}\label{thm:help}
Let $B$ be an $n\times n$ (square) matrix over \Z with $n\in\N$, and let $\delta=\gcd B$. 
Let $\phi$ be an automorphism of $\cok B$, let $\psi$ be an automorphism of $\ker B$, and let $M$ be any $n\times n$ matrix over \Z inducing $\phi$, \ie, $\phi([x])=[Mx]$ for all $x\in\Z^n$. 

Then $\det(M)\equiv\pm 1 \pmod \delta$ if and only if there exist $n\times n$ invertible (\GL) matrices $U$ and $V$ over \Z such that $UBV=B$ and $U$ is inducing $\phi$ and $V^{-1}$ is inducing $\psi$.
\end{theorem}
\begin{proof}
The only thing that does not follow from \cite[Theorem~4.4]{MR1990568} is that we can choose the \GL-equivalence $(U,V)$ such that it also induces the right automorphism on $\ker B$. For this, it is clear that we may assume that $B$ is its own Smith normal form (just like in the proof of \cite[Theorem~4.4]{MR1990568}). 
We use \cite[Theorem~4.4]{MR1990568} to get a \GL-equivalence $(U,V')\colon B\rightarrow B$ that induces $\phi$ on $\cok B$. 
The matrix $V'^{-1}$ induces an automorphism $\psi'$ of $\ker B$. 
Now we will find a \GL-equivalence $(I,V'')\colon B\rightarrow B$ that induces $\psi\circ\psi'^{-1}$ on $\ker B$ --- then $(U,V'V'')$ is a \GL-equivalence that induces $\phi$ on $\cok B$ and $\psi$ on $\ker B$. Now, the automorphism $\psi\circ\psi'^{-1}$ on $\ker B$ uniquely determines what $V''^{-1}$ should be on the lower right block matrix (where we write the matrices as $2\times 2$ block matrices according to the nonzero respectively zero part of the diagonal of $B$). Let $V''^{-1}$ be the block diagonal matrix that has this matrix as lower right block matrix and the identity as the upper left block matrix.
\end{proof}

Now, we let
$$\calP_{\min}=\setof{i\in\calP}{ j\preceq i\Rightarrow i=j}.$$
Using the above result, we get the following stronger version of \cite[Theorem~4.5]{MR1990568}:

\begin{theorem}[{Strengthening of \cite[Theorem~4.5]{MR1990568}}] \label{thm:BH-4.5-B}
Let $\mathbf{n}=(n_i)_{i\in\calP}$ be a multiindex with $n_i\neq 0$, for all $i\in\calP$. 
Suppose $B$ and $B'$ are matrices in \MPZ[\mathbf{n}] with corresponding diagonal
blocks equal, and $\kappa\colon K(B)\rightarrow K(B')$ is a $K$-web isomorphism. 
Suppose that for each $i\in\calP_{\min}$, we have an automorphism $\psi_i\colon\ker B\{i\}\rightarrow\ker B'\{i\}$. 
Then there exist matrices $U,V\in\GLZ[\mathbf{n}]$ such that we have a \GLPEe $(U, V )\colon B\rightarrow B'$ satisfying $\kappa_{(U,V )} = \kappa$ if and only if each of the
automorphisms $d_i\colon \cok B\{i\}\rightarrow \cok B'\{i\}$ defined by $\kappa$ are \GL-allowable --- 
moreover, the \GLPEe can always be chosen such that $V^{-1}\{i\}$ induces $\psi_i$ for each $i\in\calP_{\min}$.
\end{theorem}
\begin{proof}
The only thing that does not follow from \cite[Theorem~4.5]{MR1990568}, is that we can choose the \GLPEe $(U,V)$ such that it also induces the right automorphisms on $\ker B\{i\}$, $i\in\calP_{\min}$. 
We choose a \GLPEe $(U,V')$ according to \cite[Theorem~4.5]{MR1990568}, so that it induces the given $K$-web isomorphism. 
For each $i\in\calP_{\min}$, this gives an automorphism $\psi_i'$ of $\ker B\{i\}$. 
Now choose \GL-equivalences $(I,V_i'')$ of $B\{i\}$ according to the proof of Theorem~\ref{thm:help} so that $V_i''^{-1}$ induces $\psi_i\circ\psi_i'^{-1}$ for each $i\in\calP_{\min}$.
Let $V''$ be the block matrix that is the identity matrix everywhere except that 
$V''\{i\}=V_i''$ for every $i\in\calP_{\min}$. 
It is straightforward to verify that $(I,V'')$ is a \GLPEe from $B'$ to $B'$, and that 
$(U,V'V'')$ induces $\kappa$ and the desired automorphisms on the minimal components.
\end{proof}

Together with Theorem~\ref{thm:help} (using \cite[Proposition 4.1(1)]{MR1990568} and the discussion right after), 
this gives us the following stronger version of \cite[Corollary~4.7]{MR1990568}. 

\begin{corollary}[{Strengthening of \cite[Corollary~4.7]{MR1990568}}] \label{cor:BH-4.6-B}
Let $\mathbf{n}=(n_i)_{i\in\calP}$ be a multiindex with $n_i\neq 0$, for all $i\in\calP$. 
Suppose $B$ and $B'$ are matrices in $\MPZ[\mathbf{n}]$ with $\gcd B\{i\}=\gcd B'\{i\}\in\{0,1\}$ for all $i\in\calP$.
Then for any $K$-web isomorphism $\kappa\colon K(B)\rightarrow K(B')$ and any family of isomorphisms $\psi_i\colon\ker B\{i\}\rightarrow\ker B'\{i\}$, $i\in\calP_{\min}$, 
there exist matrices $U,V\in\GLZ[\mathbf{n}]$ such that we have a \GLPEe $(U, V )\colon B\rightarrow B'$ satisfying $\kappa_{(U,V )} = \kappa$ and $V^{-1}\{i\}$ induces $\psi_i$ for each $i\in\calP_{\min}$.
\end{corollary}

This allows us to lift reduced filtered $K$-theory isomorphisms to \GLPEe{s}, given that the diagonal blocks satisfy the relevant conditions about $\gcd$ (\cf\ \citefirst[Section~\ref{I-sec:red-filtered-K-theory-K-web-GLP-and-SLP-equivalences}]). 

\section{Edge expansion}
\label{sec:edge}

In this section we will look at edge expansion, what stable isomorphism between the resulting graph \cas it induces, and describe the induced map on the reduced filtered $K$-theory in terms of an \SLPEe. 
We need these results to be able to enlarge our matrices while still keeping track of the induced maps on $K$-theory --- the results are used in the proof of Lemma~\ref{lem:enlargen}.

The results in \cite{MR2922394} are important here, and we will recall the notation and results when needed.

\begin{definition}\label{def: edge expansion homomorphism}
Let $E = ( E^0 , E^1 , r_E , s_E)$ be a graph and let $e_0 \in E^1$.  Set $v_0 = s_E ( e_0 )$.  Let $F$ be the \emph{simple expansion graph} at $e_0$ defined by 
\[
F^0 = E^0 \sqcup \{ \tilde{v}_0 \}  \qquad \text{and} \qquad F^1 = \left( E^1 \setminus \{ e_0 \}  \right) \sqcup \{ f_1 , f_2 \}
\]
where $s_F \vert_{ E^1 \setminus \{ e_0 \} } = s_E \vert_{ E^1 \setminus \{ e_0 \} }$, $r_F \vert_{E^1 \setminus \{ e_0 \}} = r_E \vert_{ E^1 \setminus \{ e_0 \} }$, $s_F ( f_1 ) = v_0$, $r_F ( f_1 ) = \tilde{v}_0 = s_F ( f_2 )$, and $r_F ( f_2 )  = r_E( e_0 )$.  
\end{definition}

For a group $G$ and a set $I$, we let $G^I$ denote the direct product over $I$ and let $G^{(I)}$ denote the direct sum over $I$. 

Let $E$ be a graph. 
By \cite[Proposition~3.8]{MR2922394}, there exists an isomorphism
\[
\chi_0^{E} \colon \cok \left((\Bsf_E^{\bullet})^{\mathsf{T}}\right) \rightarrow K_0 ( C^* ( E ) ) 
\]
given by $\chi_0^{E} ( \overline{e}_v ) = [ p_v ]_0$, where $e_v$ is the element in $\Z^{(E^0)}$ that is $1$ at the $v$'th coordinate and zero in all other coordinates and $\overline{e}_v$ is the image of $e_v$ in $\cok \left( (\Bsf_E^\bullet)^\mathsf{T}\right)$.  The preimage of the positive cone of $K_0 ( C^* (E) )$ is generated by 
\[
\{ e_v \} \cup \setof{ e_v - \sum_{ e \in F } e_{r(e)}}{\text{$v \in E^0_\mathrm{sing}$, $F \subseteq s^{-1}_E(v)$, $F$ finite} }.
\]

By \cite[Proposition~3.8]{MR2922394}, there exists an isomorphism
\[
\chi_1^{E} \colon \ker \left((\Bsf_E^\bullet)^\mathsf{T}\right)\rightarrow K_1 ( C^* ( E ) ) 
\]
given by $\chi_1^{E} ( \mathbf{x} ) = [ \mathsf{U}_\mathbf{x} ]_1$.  See the paragraph before \cite[Fact~3.6]{MR2922394} for the definition of $\mathsf{U}_\mathbf{x}$.  We will describe $( \chi_1^{E} )^{-1}$.  To do this, we recall the proof of \cite[Proposition~3.8]{MR2922394}.  

Let $\mathcal{T}(E)$ be the Toeplitz algebra of $E$, \ie, $\mathcal{T}(E)$ is the universal \ca generated by a set of mutually orthogonal projections $\setof{ q_v }{ v \in E^0 }$ and a set $\setof{ t_e }{ e \in E^1 }$ of partial isometries satisfying the relations
\begin{itemize}
\item $t_e^*t_f=0$ if $e,f\in E^1$ and $e\neq f$,
\item $t_e^* t_e = q_{r(e)}$ for all $e \in E^1$, and, 
\item $t_e t_e^* \leq q_{ s(e) }$ for all $e \in E^1$.
\end{itemize}
Whenever we have a set of mutually orthogonal projections $\setof{q_v}{v\in E^0}$ and 
and a set of partial isometries $\setof{ t_e }{ e \in E^1 }$ in a \ca satisfying the above relations, then we call these elements a \emph{Toeplitz-Cuntz-Krieger $E$-family}.
Note that there exists a canonical surjective \starhomo $\pi_E \colon \mathcal{T} (E) \rightarrow C^* (E)$ such that $\pi_E ( q_v ) = p_v$ and $\pi_E ( t_e ) = s_e$ for all $v \in E^0$ and $e \in E^1$, where $\setof{ p_v, s_e }{ v \in E^0, e \in E^1 }$ is a Cuntz-Krieger $E$-family generating $C^* (E)$.  

As explained in the proof of \cite[Proposition~3.8]{MR2922394}, there exist isomorphisms \fctw{\kappa_E}{K_0 ( \ker( \pi_E ) )}{\Z^{(E^0_\mathrm{reg})}} and \fctw{\lambda_E}{K_0 ( \mathcal{T} (E) )}{\Z^{(E^0)}} such that the diagram 
\[
\xymatrix{
0 \ar[r] & K_1 ( C^* (E) ) \ar[r]^-{\partial_1^E} & K_0 ( \ker( \pi_E ) ) \ar[d]_-{\kappa_E} \ar[r]^-{(\iota_E)_*} & K_0 ( \mathcal{T} (E) ) \ar[r]^-{ (\pi_E)_*
 } \ar[d]^-{\lambda_E} & K_0 ( C^* ( E) ) \ar[r] & 0 \\
 & & \Z^{(E^0_\mathrm{reg})} \ar[r]_-{ (\Bsf_E^\bullet)^\mathsf{T} }  & \Z^{(E^0)} & & }
\]
commutes, with the top row being the exact sequence in $K$-theory induced by
\[
\xymatrix{
0 \ar[r] & \ker ( \pi_E ) \ar[r]^-{\iota_E} & \mathcal{T} (E) \ar[r]^-{ \pi_E } & C^* (E) \ar[r] & 0.
}
\]
Moreover, $\lambda_E ( [ q_v ]_0 ) = e_v$ and 
\[
\kappa_E \left( \left[ q_w - \sum_{ s(e) = w } t_e t_e^* \right] \right) = e_w  
\]
for all $v \in E^0$ and $w \in E_\mathrm{reg}^0$.  By the proof of \cite[Proposition~3.8]{MR2922394}, $(\chi_1^{E})^{-1} = \kappa_E \circ \partial_1^E$.

Suppose $E$ is a graph such that $\Bsf_E\in\MPZc$.  For a saturated, hereditary subset $H$ of $E^0$, we let $\mathfrak{J}_H$ denote the ideal in $C^*(E)$ generated by the vertex projections $\setof{ p_v }{ v \in H }$.  For saturated hereditary subsets $H_0\subseteq H\subseteq E^0$, we now show how to identify $K_{0} ( \mathfrak{J}_{H } / \mathfrak{J}_{ H_{0} } )$ and $K_{1} ( \mathfrak{J}_{H } / \mathfrak{J}_{H_{0} } )$ with the cokernel and kernel of a certain submatrix of $(\Bsf_{E}^\bullet)^{\mathsf{T}}$.  If $B$ is a matrix whose rows and columns are indexed by a set $S$ and $S_{0}$ is a subset of $S$, then we let $B\langle S_{0} \rangle$ denote the principal submatrix of $B$ whose $(i,j)$ entry, for $i, j \in S_{0}$, is the $(i,j)$ entry of $B$.  

Let $H$ be a saturated hereditary subset of $E^{0}$.  Let $E_{H}$ be the subgraph of $E$ given by $E_{H}^0 = H$, $E_{H}^{1} = s^{-1}_E (H)$, and the source and range maps of $E_{H}$ are the maps inherited from $E$.  Note that each saturated hereditary subset of $E_{H}^0$ is of the form $H_{0} \subseteq H$, with $H_{0}$ a saturated hereditary subset of $E^{0}$.  Let $H_{0}$ be a saturated hereditary subset of $E^{0}$ with $H_{0} \subseteq H$.  Let $E_{ H \setminus H_{0}}$ be the subgraph of $E$ given by $E_{ H \setminus H_{0}}^{0} = H \setminus H_{0}$ and $E_{ H \setminus H_{0}}^{1} = s^{-1}_E(H) \cap r^{-1}_E ( H \setminus H_{0} )$, and the source and range maps of $E_{H \setminus H_{0}}$ is the restriction of $s_{E}$ and $r_{E}$, respectively.  Note that $E_{ H \setminus H_{0}}$ is also a subgraph of $E_{H}$.  Then there exists an injective \starhomo $\Phi_{ H \setminus H_{0} } \colon C^{*} ( E_{ H \setminus H_{0} } ) \rightarrow \mathfrak{J}_{H} / \mathfrak{J}_{H_{0}}$ such that $\Phi_{ H \setminus H_{0} } ( p_{v, E_{ H \setminus H_{0} } } ) = p_{v, E} + \mathfrak{J}_{H_{0}}$ and $\Phi_{ H \setminus H_{0} } ( s_{e, E_{ H \setminus H_{0} } } ) = s_{e, E} + \mathfrak{J}_{H_{0}}$, and $\Phi_{  H \setminus H_{0} } ( C^{*} ( E_{ H \setminus H_{0} } ) ) = \overline{P}_{ H \setminus H_{0} } \left( \mathfrak{J}_{H} / \mathfrak{J}_{H_{0}} \right) \overline{P}_{ H \setminus H_{0} }$, where $\overline{P}_{ H \setminus H_{0} } = \sum_{ v \in H \setminus H_{0} } ( p_{v,E} + \mathfrak{J}_{H_{0}} )$.  Since $\overline{P}_{ H \setminus H_{0} }  \left( \mathfrak{J}_{H} / \mathfrak{J}_{H_{0}} \right) \overline{P}_{ H \setminus H_{0} }$ is full in $\mathfrak{J}_{H} / \mathfrak{J}_{H_{0}}$, by \cite{MR0454645} (see also \cite[Corollary~3.3]{MR2379290}), $K_{*} ( \Phi_{H \setminus H_{0}} )$ is an isomorphism.  Note that $\Asf_{E_{H \setminus H_{0}}} = (\Asf_{E})\langle H \setminus H_{0} \rangle$ and $\Bsf_{E_{H \setminus H_{0}}}^{\bullet}$ is the matrix obtained from $\Bsf_{E}^\bullet$ by removing all rows and columns whose index does not belong to $H \setminus H_{0}$.  Then 
\[
\chi_{0}^{H \setminus H_{0}} := K_{0} ( \Phi_{H \setminus H_{0}} ) \circ \chi_{0}^{E_{H \setminus H_{0} }} \colon  \cok \left( (\Bsf_{E_{H \setminus H_{0}}}^{\bullet} )^{\mathsf{T}} \right) \rightarrow K_{0} (\mathfrak{J}_{H} / \mathfrak{J}_{H_{0}}) 
\]   
and 
\[
\chi_{1}^{H \setminus H_{0}} := K_{1} ( \Phi_{H \setminus H_{0}} ) \circ \chi_{1}^{E_{H \setminus H_{0}}} \colon  \ker \left( (\Bsf_{E_{H \setminus H_{0}}}^{\bullet} )^{\mathsf{T}} \right) \rightarrow K_{1} (\mathfrak{J}_{H} / \mathfrak{J}_{H_{0}} ). 
\]  
are isomorphisms. See also \cite[Proposition~3.4]{MR1988256}.

\begin{lemma}\label{lem: edge expansion homomorphism}
Let $E = ( E^0 , E^1 , r_E , s_E)$ be a graph and let $e_0 \in E^1$.  Let $F$ be the simple expansion graph at $e_0$.  Then there exists a \starhomo $\Psi \colon \mathcal{T} (E) \rightarrow \mathcal{T} (F)$ and an injective \starhomo $\overline{\Psi} \colon C^* (E) \rightarrow C^* (F)$ such that 
\begin{enumerate}[(1)] 
\item \label{lem: edge expansion homomorphism-1}
$\pi_F \circ \Psi = \overline{\Psi} \circ \pi_E$ (in particular, $\Psi \vert_{ \ker( \pi_E ) }$ is a \starhomo from $\ker( \pi_E )$ to $\ker( \pi_F )$), 
\item \label{lem: edge expansion homomorphism-2}
$\overline{ \Psi } ( C^* ( E ) ) = P C^* (F) P$, where $P = \sum_{ v \in E^0 } p_{v, F }$ is an element of the multiplier algebra of $C^*(F)$, 
\item \label{lem: edge expansion homomorphism-3}
$\Psi ( q_{v,E} ) = q_{v,F}$ for all $v \in E^0$, and, 
\item \label{lem: edge expansion homomorphism-4}
$\Psi ( t_{e,E} ) = t_{e,F}$ for all $e \in E^{1} \setminus \{ e_{0} \}$,
\end{enumerate}
where $\setof{ q_{v,E}, t_{ e, E } }{ v \in E^0, e \in E^1 }$ is a Toeplitz-Cuntz-Krieger $E$-family generating $\mathcal{T}(E)$ and $\setof{ q_{v,F}, t_{ e, F } }{ v \in F^0, e \in F^1 }$ is a Toeplitz-Cuntz-Krieger $F$-family generating $\mathcal{T}(F)$.  
\end{lemma}

\begin{proof}
Define a Toeplitz-Cuntz-Krieger $E$-family in $\mathcal{T} (F)$ by $Q_v = q_{v,F}$ for all $v \in E^0 \subseteq F^0$ and 
\[
T_e = 
\begin{cases}
t_{e,F} &\text{if }e \neq e_0, \\
t_{f_{1} ,F } t_{ f_{2} , F } &\text{if }e = e_0.
\end{cases}
\]
One checks that $\setof{ Q_v, T_e }{ v \in E^0 , e \in E^1 }$ is a Toeplitz-Cuntz-Krieger $E$-family in $\mathcal{T} (F)$.  Hence, there exists a unique \starhomo $\Psi \colon \mathcal{T} (E) \rightarrow \mathcal{T} (F)$ such that $\Psi ( q_{v,E} ) = Q_v$ and $\Psi ( t_{e, E} ) = T_e$.  

Set $P_v=\pi_F ( Q_v )$ and $S_e=\pi_F ( T_e )$.  Then a computation shows that $\setof{ P_v , S_e }{ v \in E^0 , e \in E^1 }$ is a Cuntz-Krieger $E$-family in $C^* (F)$.  Thus, $\Psi$ induces a unique \starhomo $\overline{\Psi} \colon C^* (E) \to C^* (F)$ such that $\overline{\Psi} ( p_{v,E} ) = P_v$ and $\Psi ( s_{e,E} ) = S_e$.  By \cite[Theorem~1.2]{MR1914564}, $\overline{\Psi}$ is injective.  Since $P C^*(F) P$ is equal to the closed linear span of elements of the form $s_{\mu, F } s_{\nu, F}^*$ with $r_F( \mu ) = r_F( \nu )$ and $s_F ( \mu ), s_F( \nu ) \in E^0$ and since $s_{f_1, F } s_{f_1, F}^* = s_{ f_1 f_2, F} s_{f_1f_2, F}^*$, $\overline{ \Psi } ( C^* (E) ) =  P C^* (F) P$.  Thus, \ref{lem: edge expansion homomorphism-1}, and \ref{lem: edge expansion homomorphism-2} hold, and by construction \ref{lem: edge expansion homomorphism-3} and \ref{lem: edge expansion homomorphism-4} hold.
\end{proof}

\begin{remark}\label{rem: edge expansion homomorphism}
Suppose in Lemma~\ref{lem:  edge expansion homomorphism} that $e_0$ is an edge on a cycle based at $v_0:=s_E(e_0)$ and that $\Bsf_E\in\MPZc$. 
Assume that $j\in\calP$ corresponds to the component containing $v_0$. 
Then for the simple expansion graph $F$ at $e_0$, we have that $\Bsf_F\in\MPZc[(\mathbf{m}+\mathbf{e}_j)\times(\mathbf{n}+\mathbf{e}_j)]$, where $\mathbf{e}_j$ is the multiindex that is $1$ in the $j$th coordinate and zero elsewhere. 
Moreover, there is a one-to-one correspondence between saturated hereditary subsets of $E^0$ and saturated hereditary subsets of $F^0$ given by 
\[
H \mapsto 
\widetilde{H} = 
\begin{cases}
H &\text{if }v_0 \notin H, \\
H \cup \{ \tilde{v}_0 \} &\text{if }v_0 \in H.
\end{cases}
\]

Let $\overline{\Psi}_E \colon C^*( E) \rightarrow C^* (F)$ be the \starhomo given in Lemma~\ref{lem: edge expansion homomorphism} for the graph $E$.  Since $\overline{ \Psi }_E$ maps $\mathfrak{J}_H$ to $\mathfrak{J}_{\widetilde{H}}$ for all hereditary and saturated subsets $H\subseteq E^0$, $\overline{\Psi}_E$ induces a \starhomo, which we again denote by $\overline{ \Psi }_E$, from $\mathfrak{J}_H / \mathfrak{J}_{H_0 }$ to $\mathfrak{J}_{ \widetilde{H} } / \mathfrak{J}_{ \widetilde{H}_0 }$. 

Suppose $H$ is a saturated hereditary subset of $E^0$ containing $v_0$ and $H_0$ is a saturated hereditary subset of $E^0$ not containing $v_0$.  Then $F_{ \widetilde{H} \setminus \widetilde{H}_0}$ is the simple expansion of the graph $E_{H \setminus H_0 }$ at $e_0$.  Moreover, $ \overline{\Psi}_E \circ \Phi_{H \setminus H_0}^E = \Phi_{\widetilde{H} \setminus \widetilde{H}_0 }^F \circ \overline{\Psi}_{E_{H \setminus H_0} }$.    Here $\Phi_{H \setminus H_0 }^E$ is the embedding of $C^*(E_{H \setminus H_0 } )$ into $\mathfrak{J}_H / \mathfrak{J}_{H_0}$, $\Phi_{\widetilde{H}\setminus\widetilde{H}_0}^F$ is the embedding of $C^* ( F_{ \widetilde{H} \setminus \widetilde{H}_0 } )$ into $\mathfrak{J}_{ \widetilde{H} } / \mathfrak{J}_{ \widetilde{H}_0 }$, and $\overline{\Psi}_{E_{H \setminus H_0 } }$ is the \starhomo from $C^*( E_{H \setminus H_0 } )$ to $C^* ( F_{ \widetilde{H} \setminus \widetilde{H}_0} )$ given in Lemma~\ref{lem: edge expansion homomorphism} for the graph $E_{H \setminus H_0 }$.
\end{remark}

We now describe how symbol expansion gives rise to an \SLPEe and a \calP-equivariant stable isomorphism that induces the same map on the filtered $K$-theory as the \SLPEe. This will be used in the proof of Lemma~\ref{lem:enlargen}.

\begin{proposition}\label{prop:edge-expansion}
Let $E$ be a graph with finitely many vertices, and let $v_0$ be a vertex in $E^0$ that supports a cycle $\mu = \mu_1 \cdots \mu_n$. 
Assume that $\Bsf_E\in\MPZc$, and that $v_0$ belongs to the block $j\in\calP$. 
Let $F$ be the graph in Definition~\ref{def: edge expansion homomorphism} with $e_0 = \mu_1$ (\ie, the graph is obtained with a simple expansion of $e_0$).  Then $E$ and $F$ are move equivalent, denoted by $E\Meq F$, as defined in \citefirst[Definition~\ref{I-def:graph-equivalences}].

Moreover, $\Bsf_{F}$ is an element of $\MPZc[(\mathbf{m}+\mathbf{e}_j)\times(\mathbf{n}+\mathbf{e}_j)]$ and there exist $U\in\SLPZ[\mathbf{m}+\mathbf{e}_j]$ and $V\in\SLPZ[\mathbf{n}+\mathbf{e}_j]$ that are the identity matrices everywhere except for the $j$'th diagonal block, such that $(U,V)$ is an \SLPEe from $-\iota_{\mathbf{e}_j}(-\Bsf_{E}^\bullet)$ to  $\Bsf_{F}^\bullet$ where $\iota_{\mathbf{e}_j}$ is the embedding of $\MPZ$ into $\MPZ[(\mathbf{m}+\mathbf{e}_j)\times(\mathbf{n}+\mathbf{e}_j)]$ defined in \citefirst[Definition~\ref{I-def:iotar}], and such that there exists a \calP-equivariant isomorphism $\Phi$ from $C^{*}(E) \otimes \K$ to $C^{*}(F) \otimes \K$ satisfying $\FKR(\calP;\Phi)=\FKRs(U,V)$. 
\end{proposition}

\begin{proof}
We prove the case where $v_{0}$ is a regular vertex.  The case $v_{0}$ is an infinite emitter is similar.  We index the columns and rows of $-\iota_{\mathbf{e}_j}(-\Bsf_{E}^\bullet)$ with $F^{0}$.  The $j$'th component is indexed by $S_{0} = (H \setminus H_{0}) \sqcup \{ \widetilde{v}_{0} \}$ for the saturated hereditary subsets $H=\overline{H(v_0)}$ and $H_{0}$ of $E^{0}$ where $H_{0}$ is the largest proper saturated hereditary subset of $H$ (\cf\ \citefirst[
Lemma~\ref{I-lem:structure-2}\ref{I-lem:structure-2-3}]).  Let $U$ be the identity matrix except for the $j$'th diagonal block and the $U\{j\}$ block is the basic elementary matrix that acts on left by subtracting row $\tilde{v}_{0}$ from row $v_{0}$.  Let $V$ be the identity matrix except for the $j$'th diagonal block and the $V\{j\}$ block is the basic elementary matrix that acts on the right by subtracting column $\tilde{v}_{0}$ from column $r( \mu_{1} )$.  A computation shows that $(U, V)$ is an \SLPEe from $-\iota_{\mathbf{e}_j}(-\Bsf_{E}^\bullet)$ to  $\Bsf_{F}^\bullet$.

Clearly, $V^{\mathsf{T}} (e_{v} ) = e_{v}$ for all $v\in E^{0}$.  Moreover, for every $v\in F_\mathrm{reg}^0$ 
$$(U^{\mathsf{T}})^{-1} \left( e_{v} \right) 
= \begin{cases}
e_{v}, &\text{if }v\neq v_0, \\
e_{v_0}+e_{\widetilde{v}_0}, &\text{if }v=v_0.
\end{cases}$$

Let $\overline{\Psi} \colon C^* (E) \rightarrow C^* (F)$ be the \starhomo given in Lemma~\ref{lem: edge expansion homomorphism}.  By Remark~\ref{rem: edge expansion homomorphism}, we have that $\overline{\Psi}$ is a $\calP$-equivariant homomorphism.  We now show that $\overline{\Psi}$ induces $\FKR ( U, V)$.  Since $\overline{ \Psi } ( p_{v,E} ) = p_{v,F}$ and $V^{\mathsf{T}} (e_{v} ) = e_{v}$ for all $v\in E^{0}$, $\overline{\Psi}$ induces the desired maps on $K_0$.  Let $S \subseteq F^{0}$ represent the $i$'th component with $i \neq j$.  Since $v_{0} \notin S$, we have that $e_{0} \notin s^{-1} (S) \cap r^{-1} (S)$.  Thus, $\overline{\Psi} ( p_{v,E} ) = p_{v,F}$ and $\overline{\Psi} ( s_{e,E} ) = s_{ e, F }$ for all $v \in S$ and $e \in s^{-1} ( S ) \cap r^{-1}( S )$.  By Remark~\ref{rem: edge expansion homomorphism}, we have that $\overline{\Psi}$ induces the identity on the $K_1$-group of the $i$'th component.

We now show that $\overline{\Psi}$ induces $\FKR ( U, V)$ on the $j$'th component.  By Remark~\ref{rem: edge expansion homomorphism}, it is enough to show this for $E$ with $U = U\{j\}$ and $V = V \{ j \}$.  Let $\Psi \colon \mathcal{T} (E) \rightarrow \mathcal{T} (F)$ be the \starhomo given in Lemma~\ref{lem: edge expansion homomorphism}.  Let $\mathbf{x} = \sum_{ v \in E^{0}_{\mathrm{reg} } } x_{v} e_{v} \in \ker \left( ( \Bsf_{E}^{\bullet} )^{\mathsf{T}} \right)$.  By the proof of \cite[Proposition~3.8]{MR2922394},
\[
\partial_{1}^{E} ([ \mathsf{U}_{\mathbf{x}} ]_{1}) = \sum_{ x_{v} \neq 0} x_v\left[ q_{v,E} - \sum_{ s_{E}(e) = v } t_{e,E} t_{e,E}^{*} \right]_{0}.
\]
Therefore 
\begin{align*}
 (\partial_{1}^{F} \circ K_{1} ( \overline{\Psi} ) ) ( [ \mathsf{U}_{\mathbf{x}} ]_{1} ) 
 &= (K_{0} ( \Psi ) \circ \partial_{1}^{E})([ \mathsf{U}_{\mathbf{x}} ]_{1}) \\
 &= \sum_{ x_{v} \neq 0} x_{v} \left[ \Psi (q_{v,E}) - \sum_{ s_{E}(e) = v } \Psi ( t_{e,E} ) \Psi( t_{e,E} )^{*} \right]_{0}  \\
&= \sum_{ \substack{x_{v} \neq 0 \\ v \neq v_{0}  }} x_v \left[ q_{v,F} - \sum_{ s_{E}(e) = v} t_{e,F} t_{e,F}^{*} \right]_{0} \\
&\quad + x_{v_{0}} \left[ q_{v_{0} , F } - t_{f_{1} , F} t_{f_{2} , F } t_{f_{2},F}^{*} t_{f_{1},F}^{*} - \sum_{ \substack{ s_{E} (e) = v_{0} \\ e \neq e_{0} }  } t_{e, F } t_{e,F}^{*}  \right]_{0}.
\end{align*}
Note that in $K_{0} ( \ker( \pi_{F} ) )$, 
\begin{align*}
&\left[ q_{v_{0} , F } - t_{f_{1} , F} t_{f_{2} , F } t_{f_{2},F}^{*} t_{f_{1},F}^{*} - \sum_{ \substack{ s_{E} (e) = v_{0} \\ e \neq e_{0} }  } t_{e, F } t_{e,F}^{*}  \right] \\
&=  \left[ q_{v_{0} , F } - t_{f_{1} , F} t_{f_{1} , F}^{*} - \sum_{ \substack{ s_{E} (e) = v_{0} \\ e \neq e_{0} }  } t_{e, F } t_{e,F}^{*}  \right]_{0}  + \left[  t_{f_{1} , F} t_{f_{1} , F}^{*}  -  t_{f_{1} , F} t_{f_{2} , F } t_{f_{2},F}^{*} t_{f_{1},F}^{*}  \right]_{0} \\
&= \left[ q_{v_{0} , F } - \sum_{ s_{F} (e) = v_{0}   } t_{e, F } t_{e,F}^{*}  \right]_{0}  + \left[  t_{f_{1} , F} ( q_{ \tilde{v}_{0} , F } -   t_{f_{2} , F } t_{f_{2},F}^{*} ) t_{f_{1},F}^{*}  \right]_{0} \\
&=  \left[ q_{v_{0} , F } - \sum_{ s_{F} (e) = v_{0}   } t_{e, F } t_{e,F}^{*}  \right]_{0}  + \left[  q_{ \tilde{v}_{0} , F } -   t_{f_{2} , F } t_{f_{2},F}^{*}   \right]_{0}.
\end{align*}
Hence,
\[
\kappa_{F} \circ (\partial_{1}^{F} \circ K_{1} ( \overline{\Psi} ) ) ( [ \mathsf{U}_{\mathbf{x}} ]_{1} ) = \sum_{ \substack{x_{v} \neq 0 \\ v \neq v_{0}  }} x_v e_{v} + x_{v_{0}} e_{v_{0}} + x_{v_{0}} e_{\tilde{v}_{0}} = \sum_{x_{v} \neq 0 } x_ve_{v} + x_{ v_{0} } e_{\tilde{v}_{0}}. 
\]
Thus, $K_{1} ( \overline{\Psi} )$ is induced by $(U^{\mathsf{T}})^{-1}$.  We have just shown that $\FKR(\calP; \overline{\Psi})=\FKR(U,V)$.

By Lemma~\ref{lem: edge expansion homomorphism}, $\overline{\Psi} ( C^* ( E ) ) = P C^* (F) P$, where $P = \sum_{ v \in E^0 } p_{v,F}$.  Note that $P C^* (F) P$ is full in $C^* (F)$, and by \cite{MR0454645}, there exists a partial isometry $W$ in $\mathcal{M} ( C^* (F) \otimes \K )$ such that $W^* W = P \otimes 1_{ \mathbb{B} ( \ell^2 ) }$ and $WW^* = 1_{ \mathcal{M} ( C^* (F) ) } \otimes 1_{ \mathbb{B} ( \ell^2 ) }$, where $\mathcal{M}( C^*(F) \otimes \K )$ and $\mathcal{M} ( C^*(F))$ are the multiplier algebras of $C^*(F) \otimes \K$ and $C^*(F)$, respectively.  Since $\overline{\Psi}$ is a \calP-equivariant homomorphism, it follows from Remark~\ref{rem: edge expansion homomorphism} that $\Phi = \mathrm{Ad} ( W ) \circ (\overline{\Psi} \otimes \id_{\K}) \colon C^*(E) \otimes \K \rightarrow C^* (F) \otimes \K$ is a \calP-equivariant isomorphism.  By \cite{MR2379290}, $\FKR(\calP;\Phi) =  \FKR ( \calP ; \overline{\Psi} \otimes \id_{\K} )$ and hence $\FKR(\calP;\Phi)=\FKRs(U,V)$.
\end{proof}

\section{Stable isomorphisms induced by flow equivalences and their $K$-theory}\label{sec:toke}

Following the work of Boyle in \cite{MR1907894}, we will break \SLPEe{s} into \SLPEe{s} induced by basic elementary matrices. 
It is well-known that such \SLPEe{s} induced by basic elementary matrices induce stable isomorphisms. 
They also induce isomorphisms between the $K$-webs in a canonical way (and between the reduced filtered $K$-theories of the corresponding graph \cas). 

In this section we will show that we can choose the stable isomorphisms above in such a way that the above operations commute. 

\begin{proposition}\label{prop:toke}
Let $E$ and $F$ be graphs such that $\Bsf_E$ and $\Bsf_F$ are elements of $\MPZccc$. 
Let $u,v\in E^0$ with $u\neq v$ and $\Asf_{E}(u,v)>0$ and let $W\in\SLPZ[\mathbf{n}]$ denote the positive basic elementary matrix with $W(u,v)=1$ and equal to the identity everywhere else. 

Assume that $W\Bsf_E=\Bsf_{F}$ and the vertex $v$ in $E$ is regular. 
Then $(W^\bullet,I)$ is an \SLPEe from $\Bsf_E^\bullet$ to $\Bsf_{F}^\bullet$, where $W^\bullet$ is the matrix we get from $W$ by removing all columns and rows corresponding to singular vertices of $E$. 
Moreover, there exists a $\calP$-equivariant isomorphism $\Phi$ from $C^*(E)$ to $C^*(F)$ such that $\FKR(\calP;\Phi)=\FKR(W^\bullet,I)$.

Now assume instead that $\Bsf_E W=\Bsf_{F}$ and the vertex $u$ emits at least two edges. 
Then $(I,W)$ is an \SLPEe from $\Bsf_E^\bullet$ to $\Bsf_{F}^\bullet$. Moreover, there exists a $\calP$-equivariant isomorphism $\Phi$ from $C^*(E)\otimes\K$ to $C^*(F)\otimes\K$ such that $\FKR(\calP;\Phi)=\FKRs(I,W)$. 
\end{proposition}

\begin{proof}
Throughout the proof, $\pi_G$ will denote the canonical surjective homomorphism from $\mathcal{T} (G)$ to $C^* (G)$, where $G$ is a graph. 
Also, $q_w^G$ and $t_e^G$ will denote the canonical generators of  $\mathcal{T} (G)$ while $p_w^G$ and $s_e^G$ denotes the canonical generators of $C^*(G)$. 
See Section~\ref{sec:edge} for more details and notation. 

\textbf{Part one:} We first prove the row addition part of the proposition.  Let $f \in s^{-1}(u) \cap r^{-1}(v)$ which exists since $\Asf_E (u,v) > 0$.  Note that $F$ can be described as follows:  $F^0 = E^0$, 
\[
F^1 = \left( E^1 \setminus \{ f \} \right) \sqcup \{ \overline{e} : e \in s_E^{-1} ( v ) \}
\]
where $s_F ( e ) = s_E (e)$ and $r_F (e) = r_E (e)$ for all $e \in E^1 \setminus \{f\}$, $s_F ( \overline{e} ) = u$ and $r_F( \overline{e} ) = r_E (e)$.

We now define a Toeplitz-Cuntz-Krieger $F$-family in $\mathcal{T} (E)$.  Set $Q_w = q_w^E$ and 
\[
T_g = 
\begin{cases}
t_e^E 	&\text{if }g = e \in E^1 \setminus \{ f \} \\
t_{fe}^E 	&\text{if }g = \overline{e}\text{ for some }e\in s_E^{-1}(v).
\end{cases}
\] 
A computation shows that $\{ Q_w , T_g : w \in F^0, g \in F^1 \}$ is a Toeplitz-Cuntz-Krieger $F$-family such that 
\begin{equation}\label{eq: adding rows}
 Q_w - \sum_{ g \in s^{-1}_F ( w ) } T_g T_g^* \in \ker ( \pi_E )
\end{equation}
for all $w\in F_{\mathrm{reg}}^0$.  Then, there exists a \starhomo $\Psi \colon \mathcal{T} (F) \rightarrow \mathcal{T} (E)$ such that $\Psi ( q_w^F) = Q_w$ and $\Psi ( t_g^F ) = T_g$. Equation~\eqref{eq: adding rows} implies that there exists a (unique) \starhomo $\overline{\Psi} \colon C^* (F) \rightarrow C^* (E)$ such that $\pi_E \circ \Psi = \overline{\Psi} \circ \pi_F$.

Note that the identity map on $E^0 = F^0$ induces a lattice isomorphism between the lattice of saturated hereditary subsets of $F^0$ and the lattice of saturated hereditary subsets of $E^0$.  Since
\[
\overline{\Psi} ( p_w^F ) = ( \overline{ \Psi } \circ \pi_F)( q_w^F ) = ( \pi_E \circ \Psi ) ( q_w^F ) = \pi_E ( Q_w ) = \pi_E ( q_w^E ) = p_w^E,
\]
$\overline{\Psi} ( \mathfrak{J}_H^F ) \subseteq \mathfrak{J}_H^E$ for all saturated hereditary subsets $H$ of $F^0$.  We have just shown that $\overline{\Psi}$ is a $\mathcal{P}$-equivariant homomorphism.

By \cite[Theorem~1.2]{MR1914564}, $\overline{ \Psi }$ is injective.  Note that the only generator of $C^* (F)$ that is not in the image of $\overline{ \Psi }$ is $s_f^E$.  Since 
\[
s_f^E = s_f^E p_v^E = s_f^E \sum_{ e \in s^{-1}_E (v) } s_e^E (s_e^E)^* = \sum_{ e \in s^{-1}_E ( v ) } \pi_E(T_{\overline{e} } T_{e}^*) \in \overline{ \Psi } ( C^* (F) ),
\]
it follows that $\overline{\Psi}$ is surjective.  Therefore, $\overline{\Psi} ( \mathfrak{J}_H^F ) = \mathfrak{J}_H^E$ for all saturated hereditary subsets $H$ of $F^0$, and, consequently $\overline{\Psi}$ is a $\mathcal{P}$-equivariant isomorphism (since $\overline{\Psi}$ is a \stariso).

Let $\Phi \colon C^* (E) \rightarrow C^* (F)$ be the inverse of $\overline{\Psi}$.  To prove that $\FKR ( \calP ; \Phi ) = \FKR ( W^\bullet, I )$, it is enough to prove that $\FKR ( \calP ; \overline{\Psi} ) = \FKR ( (W^\bullet)^{-1}, I )$.

Since $\overline{\Psi} ( p_w^F) = p_w^E$, $\overline{ \Psi }$ sends $[ p_w^F ]_0 \in K_0 ( \mathfrak{J}_H^F )$ to $[ p_w^E ]_0 \in K_0 ( \mathfrak{J}_H^E)$ which implies that $\overline{ \Psi }$ induces the desired maps on $K_0$.

Now we want to show that $\overline{\Psi}$ induces the desired maps on $K_1$ (of the gauge simple subquotients of $C^* (F)$). For all components that do not contain both $u$ and $v$, it is clear that this holds (since both maps are clearly the identity map). 
In the opposite direction, if $u$ and $v$ both belong to a component, we only have to restrict to the subgraph of that component. For notational convenience, we show this for $C^* (E)$ instead, corresponding to the gauge simple case --- the general case is analogous. Let $\mathbf{x} = ( x_w )_{ w \in F^0_{\mathrm{reg}} } \in \ker \left( ( \Bsf_{F}^\bullet )^\mathsf{T} \right)$.  Note that 
\[
(W^\bullet)^\mathsf{T} (\mathbf{x} ) = \sum_{ \substack{ x_w \neq 0 \\ w \neq v } } x_w e_w + (x_v + x_u) e_v.
\]   

In $K_0 ( \ker ( \pi_E ) )$, 
\begin{align*}
&\sum_{ x_w \neq 0 } x_w \left[ \Psi ( q_w^F ) - \sum_{ g \in s_F^{-1}(w)  } \Psi ( t_g^F ) \Psi ( t_g^F)^* \right]_0 \\
&\quad = \sum_{ \substack{ x_w \neq 0  \\ w \neq u } } x_w \left[ q_w^E - \sum_{ e \in s_E^{-1} (w) } t_e^E ( t_e^E)^*\right]_0 \\
&\qquad + x_u \left[ q_u^E  - \sum_{ e \in s_E^{-1} ( u ) \setminus \{f\} } t_e^E ( t_e^E )^* - \sum_{ e\in s_E^{-1} (v) } t_{fe}^E ( t_{fe}^E )^*\right]_0. 
\end{align*}
Since $v$ is regular, $q_v^E - \sum_{ e\in s_E^{-1} (v) } t_e^E ( t_e^E )^* \in \ker( \pi_E )$.  Hence, in $K_0 ( \ker ( \pi_E ) )$,
\begin{align*}
&\left[ q_u^E  - \sum_{ e \in s_E^{-1} ( u ) \setminus \{f\} } t_e^E ( t_e^E )^* - \sum_{ e\in s_E^{-1} (v) } t_{fe}^E ( t_{fe}^E )^*\right]_0  \\
&\quad = \left[ q_u^E -  \sum_{ e \in s_E^{-1} ( u ) } t_e^E ( t_e^E )^* \right]_0 + \left[ t_f^E \left( q_v^E - \sum_{ e\in s_E^{-1} (v) } t_e^E ( t_e^E )^* \right) (t_f^E)^* \right]_0 \\
&\quad = \left[ q_u^E -  \sum_{ e \in s_E^{-1} ( u ) } t_e^E ( t_e^E )^* \right]_0 + \left[ q_v^E - \sum_{ e\in s_E^{-1} (v) } t_e^E ( t_e^E )^* \right]_0.
\end{align*}
Thus,
\begin{align*}
&\sum_{ x_w \neq 0 } x_w \left[ \Psi ( q_w^F ) - \sum_{ g \in s_F^{-1}(w) } \Psi ( t_g^F ) \Psi ( t_g^F)^* \right]_0  \\
&\quad = \sum_{ \substack{ x_w \neq 0  \\ w \neq u  } } x_w \left[ q_w^E - \sum_{ e \in s_E^{-1} (w) } t_e^E ( t_e^E)^*\right]_0 \\
&\qquad + x_u\left[ q_u^E -  \sum_{ e \in s_E^{-1} ( u ) } t_e^E ( t_e^E )^* \right]_0 + x_u\left[ q_v^E - \sum_{ e\in s_E^{-1} (v) } t_e^E ( t_e^E )^* \right]_0.
\end{align*}
So,
\begin{align*}
&(\kappa_E \circ  \partial_1^E \circ K_1 ( \overline{\Psi} ) \circ \chi_1^F ) ( \mathbf{x} ) \\
&\quad =  \left( \kappa_E \circ K_0 ( \Psi ) \circ \partial_1^F \circ \chi_1^F \right) ( \mathbf{x} ) \\
&\quad = \sum_{ \substack{ x_w \neq 0  \\ w \neq u } } x_w \kappa_E \left( \left[ q_w^E - \sum_{ e \in s_E^{-1} (w) } t_e^E ( t_e^E)^*\right]_0 \right) \\
&\qquad + x_u \kappa_E \left( \left[ q_u^E -  \sum_{ e \in s_E^{-1} ( u ) } t_e^E ( t_e^E )^* \right]_0 \right) + x_u \kappa_E \left( \left[ q_v^E - \sum_{ e\in s_E^{-1} (v) } t_e^E ( t_e^E )^* \right]_0 \right) \\
&\quad=  \sum_{ \substack{ x_w \neq 0 \\ w \neq u }  } x_w e_w + x_u e_u + x_u e_v \\
&\quad = \sum_{ \substack{ x_w \neq 0 \\ w \neq v } } x_w e_w  + ( x_u + x_v ) e_v.
\end{align*}
Hence, $\overline{\Psi}$ induces the desired maps on $K_1$.

\textbf{Part two:} We now prove the column addition part of the proposition.  Note that $F$ can be described as follows:  Take an edge $f$ from $u$ to $v$.  Set $F^0 = E^0$, 
\[
F^1 = \left( E^1 \setminus \{f\} \right) \sqcup \{ \overline{e} : e \in r_E^{-1} (u) \}
\]
with $s_F ( e ) = s_E (e)$, $r_F (e) = r_E (e)$ for all $e \in E^1 \setminus \{ f \}$, and $s_F ( \overline{e} ) = s_E (e)$ and $r_F ( \overline{e} ) = v$.

We define a Toeplitz-Cuntz-Krieger $F$-family in $\mathcal{T} ( E )$.  Set 
\[
Q_w = 
\begin{cases}
q_w^E &\text{if $w \neq u$} \\
q_u^E - t_f^E (t_f^E)^* &\text{if $w = u$}.
\end{cases}
\]
and 
\[
T_g = 
\begin{cases}
t_g^E &\text{if }g\in E^1\setminus\{f\}\text{ and }r_E ( g ) \neq u, \\
t_g^E Q_u &\text{if }g\in E^1\setminus\{f\}\text{ and }r_E ( g ) = u, \\
t_{ef}^E &\text{if }g = \overline{e}\text{ for some }e\in r_E^{-1}(u).
\end{cases}
\]
A computation shows that $\{ Q_w , T_g : w \in F^0 , g \in F^1 \}$ is a Toeplitz-Cuntz-Krieger $F$-family in $\mathcal{T} ( E )$ such that 
\begin{equation}\label{eq: adding columns}
 Q_w - \sum_{ g \in s^{-1}_F ( w ) } T_g T_g^* \in \ker ( \pi_E )
\end{equation}
for all $w \in F^0_\mathrm{reg}$.  Then, there exists a \starhomo $\Psi \colon \mathcal{T} (F) \rightarrow \mathcal{T} (E)$ such that $\Psi ( q_w^F) = Q_w$ and $\Psi ( t_g^F ) = T_g$. Equation~\eqref{eq: adding columns} implies that there exists a (unique) \starhomo $\overline{\Psi} \colon C^* (F) \rightarrow C^* (E)$ such that $\pi_E \circ \Psi = \overline{ \Psi } \circ \pi_F$.  Moreover, $\overline{ \Psi } ( C^* (F) ) = P C^* (E) P$, where $P = \sum_{ w \in E^0 } \pi_E(Q_w)$.  Note that $P C^* (E) P$ is full in $C^* (E)$.  Also, by \cite[Theorem~1.2]{MR1914564}, $\overline{\Psi}$ is injective.  Hence, $\overline{ \Psi } \colon C^* (F) \rightarrow P C^* (E) P$ is a \stariso.  

Since $P C^* (E) P$ is full in $C^* (E)$, it follows from \cite[Corollary~2.6 and its proof]{MR0454645} that there exists $V$ in the multiplier algebra $\mathcal{M}{(C^* (E) \otimes \K } )$ of $C^*(E) \otimes \K$ such that $V^{*} V = P \otimes 1_{ \mathbb{B} ( \ell^{2} ) }$ and $VV^{*} = 1_{ \mathcal{M} ( C^{*} (E)  ) } \otimes 1_{ \mathbb{B} ( \ell^{2} ) }$.  Then $\mathrm{Ad} ( V) \circ ( \overline{ \Psi }  \otimes \mathrm{id}_{ \K } ) : C^{*} (F) \otimes \K \rightarrow C^{*} (E) \otimes \K$ is a \stariso.  

We will show that $\mathrm{Ad} ( V) \circ ( \overline{ \Psi }  \otimes \mathrm{id}_{ \K } )$ is a $\mathcal{P}$-equivariant isomorphism.  We first show that $\overline{ \Psi }$ is a $\mathcal{P}$-equivariant homomorphism, which implies that $\mathrm{Ad} ( V) \circ ( \overline{ \Psi }  \otimes \mathrm{id}_{ \K } )$ is a $\mathcal{P}$-equivariant homomorphism.  Note that the identity map on $E^0 = F^0$ induces a lattice isomorphism between the lattice of saturated hereditary subsets of $F^0$ and the lattice of saturated hereditary subsets of $E^0$.  Let $H$ be a hereditary subset of $F^0$.  Let $w \in H$.  Since $\overline{\Psi} ( p_w^F)  \leq p_w^E$, we have that $\overline{ \Psi } ( p_w^F ) \in \mathfrak{J}_H^E$.  Hence, $\overline{\Psi} ( \mathfrak{J}_H^F ) \subseteq  \mathfrak{J}_H^E$.  We have just shown that $\overline{\Psi}$ is a $\mathcal{P}$-equivariant homomorphism.

Let $H$ be a saturated hereditary subset of $F^0$.  Then $\mathrm{Ad} ( V) \circ ( \overline{ \Psi }  \otimes \mathrm{id}_{ \K } )( \mathfrak{J}_H^F \otimes \K )$ is the ideal of $C^* (E) \otimes \K$ generated by $\setof{ V (\pi_E(Q_w) \otimes e_{11}) V^* }{ w \in H }$, where $e_{11}$ is a minimal nonzero projection in $\K$.  Since $V (\pi_E(Q_w) \otimes e_{11}) V^*$ is Murray-von Neumann equivalent to $\pi_E(Q_w) \otimes e_{11}$ for all $w \in H$, we have that $\mathrm{Ad} ( V) \circ ( \overline{ \Psi }  \otimes \mathrm{id}_{ \K } )( \mathfrak{J}_H^F \otimes \K )$ is the ideal of $C^* (E) \otimes \K$ generated by $\setof{ \pi_E(Q_w) \otimes e_{11} }{ w \in H }$.  Note that 
\[
p_u^E \otimes e_{11} = \left( p_u^E - s_f^E (s_f^E)^*\right) \otimes e_{11} + \left(s_f^E (s_f^E)^*\right) \otimes e_{11}
\]
and $\left(s_f^E (s_f^E)^*\right) \otimes e_{11}$ is Murray-von Neumann equivalent to $p_v^E \otimes e_{22}$, where $e_{22}$ is a minimal nonzero projection in $\K$ that is orthogonal to $e_{11}$.  Hence, $p_u^E \otimes e_{11}$ is Murray-von Neumann equivalent to $( p_u^E - s_f^E (s_f^E)^*) \otimes e_{11} + p_v^E \otimes e_{22}$ which is in the ideal of $C^* (E) \otimes \K$ generated by $\setof{\pi_E( Q_w) \otimes e_{11} }{ w \in H }$ whenever $u\in H$.  Therefore, the ideal of $C^* (E) \otimes \K$ generated by $\setof{ \pi_E( Q_w  )\otimes e_{11} }{ w \in H }$ is $\mathfrak{J}_H^E \otimes \K$.  Thus, $\mathrm{Ad} ( V) \circ ( \overline{ \Psi }  \otimes \mathrm{id}_{ \K } )( \mathfrak{J}_H^F \otimes \K ) = \mathfrak{J}_H^E \otimes \K$.  We have just shown that $\mathrm{Ad} ( V) \circ ( \overline{ \Psi }  \otimes \mathrm{id}_{ \K } )$ is a $\mathcal{P}$-equivariant isomorphism.

We now show that $\mathrm{Ad} ( V) \circ ( \overline{ \Psi }  \otimes \mathrm{id}_{ \K } )$ induces $\FKRs ( I , W^{-1} )$.  Since $V$ is an isometry in the multiplier algebra $\mathcal{M}{(C^* (E) \otimes \K } )$ of $C^*(E) \otimes \K$, it is enough to show that $\overline{\Psi}$ induces $\FKR ( I , W^{-1} )$ (\cf\ \cite[Lemma~3.4]{MR2379290}). 
Note that  
\[
\left(W^{-1}\right)^\mathsf{T} ( e_w ) =
\begin{cases}
e_w &\text{if }w \neq u, \\
e_u - e_v &\text{if }w = u.
\end{cases}
\]

Let $H$ be a saturated subset $H$ of $F^0$. It is trivial to check that the induced maps are equal on $K_0$ of $\mathfrak{J}_H^E$ whenever $u\not\in H$. So assume that $u \in H$. Then $p_u^E , s_f^E \in \mathfrak{J}_H^E$, and, consequently, we have that
\begin{align*}
\sum_{w \in H } x_w \left[  \overline{\Psi} ( p_w^F ) \right]_0 &= \sum_{ w \in H \setminus \{ u \} } x_w \left[ p_w^E \right]_0 + x_u \left[  p_u^E - s_f^E (s_f^E)^* \right]_0 \\
&= \sum_{ w \in H \setminus \{ u \} } x_w \left[ p_w^E \right]_0 + x_u \left[  p_u^E  \right]_0 - x_u \left[ s_f^E (s_f^E)^* \right]_0 \\
&= \sum_{ w \in H \setminus \{ u \} } x_w \left[ p_w^E \right]_0 + x_u \left[  p_u^E  \right]_0 - x_u \left[ (s_f^E)^* s_f^E \right]_0 \\
&= \sum_{ w \in H \setminus \{ u \} } x_w \left[ p_w^E \right]_0 + x_u \left[  p_u^E  \right]_0 - x_u \left[ p_{r_E (f) }^E \right]_0   \\
&= \sum_{ w \in H \setminus \{ u \} } x_w \left[ p_w^E \right]_0 + x_u \left[  p_u^E  \right]_0 - x_u \left[ p_{v }^E \right]_0 
\end{align*}
in $K_0 ( \mathfrak{J}_H^E)$. 
So,
\[
\left( K_0( \overline{\Psi} )  \circ \chi_0^H \right)\left( \sum_{ w \in H } x_w \overline{e}_w \right) = \sum_{w \in H } x_w \left[  \overline{\Psi} ( p_w^F ) \right]_0 =  \left( \chi_0^H\circ (W^{-1})^\mathsf{T} \right)\left( \sum_{ w \in H } x_w \overline{e}_w \right)
\]
in $K_0 ( \mathfrak{J}_H^E)$.
Thus, $\overline{\Psi}$ induces the desired maps on $K_0$.

We now show that $\overline{\Psi}$ induces the desired maps on $K_1$.  Let $w \in F^0_\mathrm{reg} = E^0_\mathrm{reg}$ and $w \neq u$.  Then
\begin{align*}
&\Psi ( q_w^F ) - \sum_{ g \in s_F^{-1} (w) } \Psi ( t_g^F (t_g^F)^* ) \\
&\quad = q_w^E - \sum_{ \substack{ e \in s_E^{-1}(w) \\ r_E (e) \neq u } } t_e^E (t_e^E)^* - \sum_{ \substack{ e \in s_E^{-1}(w) \\ r_E (e) = u } } t_e^E \left( q_u^E - t_f^E (t_f^E)^* \right)(t_e^E)^* - \sum_{ \substack{ e \in s_E^{-1} (w)  \\ r_E (e) = u } } t_e^E t_f^E (t_f^E)^* (t_e^E)^* \\
&\qquad  = q_w^E - \sum_{ \substack{ e \in s_E^{-1}(w) \\ r_E (e) \neq u } } t_e^E (t_e^E)^* - \sum_{ \substack{ e \in s_E^{-1}(w) \\ r_E (e) = u } } t_e^E (t_e^E)^*  \\
&\quad = q_w^E - \sum_{ e \in s_E^{-1} (w) } t_e^E (t_e^E)^*.
\end{align*}
If $u$ is regular, then
\begin{align*}
&\Psi ( q_u^F ) - \sum_{ g \in s_F^{-1} (u) } \Psi ( t_g^F (t_g^F)^* ) \\
&\quad = q_u^E - t_f^E (t_f^E)^*- \sum_{ \substack{ e \in  (s_E^{-1}(u) \setminus \{f \} ) \\ r_E (e) \neq u } } t_e^E (t_e^E)^* - \sum_{ \substack{ e \in s_E^{-1}(u) \\ r_E (e) = u } } t_e^E ( q_u^E - t_f^E (t_f^E)^* ) (t_e^E)^* \\
&\qquad - \sum_{ \substack{ e \in s_E^{-1} (u)  \\ r_E (e) = u } } t_e^E t_f^E (t_f^E)^* (t_e^E)^* \\
&\quad = q_u^E - \sum_{ \substack{ e \in  s_E^{-1}(u)  \\ r_E (e) \neq u } } t_e^E (t_e^E)^* - \sum_{ \substack{ e \in s_E^{-1}(u) \\ r_E (e) = u }}t_e^E(t_e^E)^* \\
&\quad = q_u^E - \sum_{ e \in s_E^{-1} (u) } t_e^E (t_e^E)^*.
\end{align*}

Now we want to show that $\overline{\Psi}$ induces the desired maps on $K_1$ (of the gauge simple subquotients of $C^* (E)$). For all components that do not contain both $u$ and $v$, it is clear that this holds (since both maps are clearly the identity map). 
In the opposite direction, if $u$ and $v$ both belong to a component, we only have to restrict to the subgraph of that component. For notational convenience, we show this for $C^* (E)$ instead, corresponding to the gauge simple case --- the general case is analogous. 
Let $\mathbf{x} = ( x_w )_{ w \in F^0_{\mathrm{reg}} } \in \ker \left( ( \Bsf_{F}^\bullet )^\mathsf{T} \right)$.  Then, in $K_0 ( \ker ( \pi_{ E} ) )$, we have that 
\[
\sum_{x_w \neq 0 } x_w \left[ \Psi ( q_w^F ) - \sum_{ e \in s_F^{-1} (w) } \Psi ( t_e^F (t_e^F)^* ) \right]_0 = \sum_{x_w \neq 0 } x_w \left[ q_w^E  - \sum_{ e \in s_E^{-1} (w) }  t_e^E (t_e^E)^* ) \right]_0. 
\]
So, 
\begin{align*}
(\kappa_E \circ  \partial_1^E \circ K_1 ( \overline{\Psi} ) \circ \chi_1^F ) ( \mathbf{x} ) &=  \left( \kappa_E \circ K_0 ( \Psi ) \circ \partial_1^F \circ \chi_1^F \right) ( \mathbf{x} ) \\
			&=  \kappa_E \left( \sum_{x_w \neq 0 } x_w \left[ q_w^E  - \sum_{ e \in s_E^{-1} (w) }  t_e^E (t_e^E)^* ) \right]_0 \right) \\
			&= \sum_{x_w \neq 0 } x_w e_w.
\end{align*}
Hence, $\overline{\Psi}$ induces the desired maps on $K_1$.

Set $\Phi = \left( \mathrm{Ad} ( V) \circ ( \overline{ \Psi }  \otimes \mathrm{id}_{ \K } ) \right)^{-1} : C^{*} (E) \otimes \K \rightarrow C^{*} (F) \otimes \K$.  Then $\Phi$ is a $\mathcal{P}$-equivariant isomorphism that induces $ \FKRs ( I , W)$.
\end{proof}

Recall from \citefirst[Definition~\ref{I-def: positive matrices}] that when $E$ is a finite graph with no sinks and no sources satisfying Condition~(K), then we write $\Bsf_E\in\MPplusZ[\mathbf{n}]$ whenever $\Bsf_E\in\MPZ[\mathbf{n}]$ and for all $i,j\in\calP$ 
\begin{enumerate}[(i)]
\item if $i\preceq j$, then $\Bsf_E\{i, j\} > 0$,
\item $n_i\geq 3$ and the Smith normal form of $\Bsf_E\{i\}$ has at least two $1$’s (and thus the rank of $\Bsf_E \{i\}$ is at least 2).
\end{enumerate} 

\begin{remark}
Let $E$ be a finite graph with no sinks and no sources satisfying Condition~(K) such that $\Bsf_E\in\MPplusZ[\mathbf{n}]$.  Then $\Bsf_E \in \MPZccc[\mathbf{n}]$ and every vertex in $E$ emits at least two edges since the diagonal entries $\Asf_E = \Bsf_E + I$ are greater than or equal to 2.
\end{remark}

\begin{lemma}\label{lem:stdform}
Let $A$ be a nondegenerate matrix that satisfies Condition~(II). 
Then there exists a finite graph $E$ with no sinks and no sources that satisfies Condition~(K) such that $\Ocal_{A}\otimes\K \cong C^{*}(E)\otimes\K$ and $\Bsf_E\in\MPplusZ[\mathbf{n}]$, for a suitable $\mathbf{n}$ and $\calP$.
\end{lemma}
\begin{proof}
The matrix $A$ defines a graph $F$ (with adjacency matrix $A$) with no sinks and no sources that satisfies Condition~(K). 
The graph \ca $C^{*}(F)$ is actually isomorphic to the Cuntz-Krieger algebra $\Ocal_A$ via a \stariso defined by $p_{i, F} \mapsto s_i s_i^*$ and $s_{e, F} \mapsto s_{s_F(e)} s_{r_F(e)} s_{r_F(e)}^*$, where $\{ p_{i, F } , s_{e, F}\}$ is a Cuntz-Krieger $F$-family generating $C^*(F)$ and $\{ s_1, \dots, s_n \}$ is a universal family of partial isometries generating $\Ocal_A$.
Thus we can work with $F$ instead. 
Now it follows from \citefirst[Lemma~\ref{I-taugivesstd}] and \citefirst[Theorem~\ref{I-thm:moveimpliesstableisomorphism}] that we can choose a finite graph $E$ such that $\Bsf_E\in\MPplusZ$ and $\Ocal_{A}\otimes\K \cong C^{*}(E)\otimes\K$. 
Since $\Bsf_E\in\MPplusZ$, the graph $E$ does not have any source. 
Note that Condition~(K) with no sinks means exactly that the temperature is constant $1$, \cf\ \citefirst[Definition~4.16 and Lemma~\ref{I-charKH}\ref{I-charKH-1}]). 
Since the temperature is clearly a stable isomorphism invariant, it follows that $E$ has no sinks and satisfies Condition~(K), as well as $\mathbf{m}=\mathbf{n}$. 
\end{proof}

\begin{lemma}\label{lem:enlargen}
Let $E$ be a finite graph with no sinks and no sources satisfying Condition~(K) such that $\Bsf_E\in\MPplusZ[\mathbf{n}]$, and let $j\in\calP$ be given. 
Then there exists a finite graph $F$ with no sinks and no sources satisfying Condition~(K) such that $\Bsf_F\in\MPplusZ[\mathbf{n}+\mathbf{e}_j]$, and such that there exists a \calP-equivariant isomorphism $\Phi$ from $C^{*}(E)\otimes\K$ to $C^{*}(F)\otimes\K$ satisfying $\FKR ( \calP ; \Phi )=\FKRs(U,V)$ for some \SLPEe $(U,V)$ from $-\iota_{\mathbf{e}_j}(-\Bsf_E)$ to $\Bsf_F$. 
\end{lemma}
\begin{proof}
Let $e_0$ be an edge that is part of a cycle that lies in the component corresponding to $j\in\calP$. 
Let $E'$ be the simple edge expansion graph at $e_0$, and let $v'$ denote the new vertex (\cf\ Definition~\ref{def: edge expansion homomorphism}). It follows from Proposition~\ref{prop:edge-expansion} that $\Bsf_{E'}\in \MPZc[\mathbf{n}+\mathbf{e}_j]$ and that there exists a \calP-equivariant isomorphism $\Phi'$ from $C^*(E)\otimes\K$ to $C^*(E')\otimes\K$ such that $\FKR(\calP,\Phi')=\FKRs(U',V')$ for some \SLPEe $(U',V')$ from $-\iota_{\mathbf{e}_j}(-\Bsf_E)$ to $\Bsf_{E'}$. The idea is to use Proposition~\ref{prop:toke} to construct a finite graph $F$ with no sinks and no sources satisfying Condition~(K) such that $\Bsf_F\in \MPplusZ[\mathbf{n}+\mathbf{e}_j]$, and a \calP-equivariant isomorphism $\Phi''$ from $C^*(E')\otimes\K$ to $C^*(F)\otimes\K$ such that $\FKR(\calP,\Phi'')=\FKRs(U'',V'')$ for some \SLPEe $(U'',V'')$ from $\Bsf_{E'}$ to $\Bsf_F$. We have that $\Bsf_{E'}\in \MPZc[\mathbf{n}+\mathbf{e}_j]$, and since $E$ has no transition states, no cyclic components, no sinks and no sources, it follows that $E'$ has no transition states, no cyclic components, no sinks and no sources, and thus that $\Bsf_{E'}\in\MPZccc[\mathbf{n}+\mathbf{e}_j]$.  
We may assume that $v_0=s(e_0)$ corresponds to the last row of the diagonal block $\Asf\{j\}$. Let $v_1=r(e_0)$. 
We will now perform a series of column and row additions on the matrix $\Asf_{E'}-I$ according to Proposition~\ref{prop:toke}. Using a column addition within block $j$, we can make sure that $v_0$ supports at least two loops. Adding column $v_0$ to $v'$, we get all entries of the column $v'$ in $j$'th diagonal block to be positive except for the last entry. Adding row $v_1$ twice to row $v'$, the $j$'th diagonal block is positive. Now we can use column additions and row additions to get all the off-diagonal blocks that are not forced by the order on \calP to be zero, to be positive, and we thus end up with a graph $F$ such that $\Bsf_F\in \MPplusZ[\mathbf{n}+\mathbf{e}_j]$. Since these are all legal column and row operations according to Proposition~\ref{prop:toke}, we get the desired result from Proposition~\ref{prop:edge-expansion}. 
\end{proof}

\section{Cuntz splice implies stable isomorphism}
\label{sec:cuntzsplice}

In \cite{arXiv:1602.03709v2}, it was proved that Cuntz splicing a graph gives rise to a graph whose \ca is stably isomorphic to the \ca of the original graph (generalizing a result in \cite{MR2270572}). This is used for instance in the proofs of \citefirst to get a classification result. 
As we are proving a strong classification result, we will need to know more about what this stable isomorphism induces on the reduced filtered $K$-theory. This will be the main aim of this section. In the end of the section, we will use the results to stitch together the proofs of the main results.

We recall some notation from \cite{arXiv:1602.03709v2} that will be used throughout the rest of this section.

\begin{notation}\label{notation:OnceAndTwice}
Let $\mathbf{E}_*$ and $\mathbf{E}_{**}$ denote the graphs: 
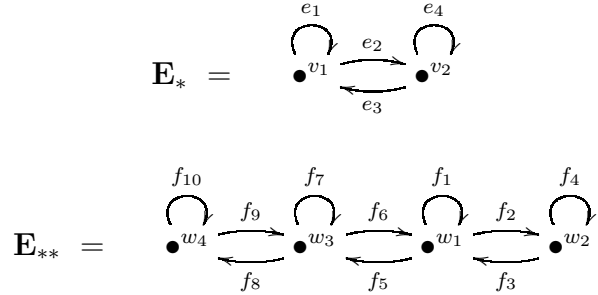
\begin{align*}
\mathbf{E}_* \  = \ \ \ \ \xymatrix{
  \bullet^{v_1} \ar@(ul,ur)[]^{e_{1}} \ar@/^/[r]^{e_{2}} & \bullet^{v_2} \ar@(ul,ur)[]^{e_{4}} \ar@/^/[l]^{e_{3}}
}
\end{align*}
\begin{align*}
\mathbf{E}_{**} \  =  \ \ \ \ \xymatrix{
	\bullet^{ w_{4} } \ar@(ul,ur)[]^{f_{10}}  \ar@/^/[r]^{ f_{9} } & \bullet^{ w_{3} } \ar@(ul,ur)[]^{f_{7}} \ar@/^/[r]^{ f_{6} }  \ar@/^/[l]^{f_{8}} &  \bullet^{w_1} 				\ar@(ul,ur)[]^{f_{1}} \ar@/^/[r]^{f_{2}} \ar@/^/[l]^{f_{5}}
	& \bullet^{w_2} \ar@(ul,ur)[]^{f_{4}} \ar@/^/[l]^{f_{3}}
	}
\end{align*}
The graph $\mathbf{E}_*$ is what we attach when we Cuntz splice. If we instead attach the graph $\mathbf{E}_{**}$, we Cuntz spliced twice. 

Let $E = ( E^{0}, E^{1} , r_{E}, s_{E} )$ be a graph and let $u$ be a vertex of $E$.
Then $E_{u, -}$ can be described as follows (up to canonical isomorphism):
\begin{align*}
E_{u,-}^{0} &= E^{0} \sqcup \mathbf{E}_{*}^{0} \\
E_{u,-}^{1} &= E^{1} \sqcup \mathbf{E}_{*}^{1} \sqcup \{ d_1, d_2 \}
\end{align*}
with $r_{E_{u,-}} \vert_{E^{1}} = r_{E}$, $s_{E_{u,-}} \vert_{ E^{1} } = s_{E}$, $r_{E_{u,-}} \vert_{\mathbf{E}_{*}^{1}} = r_{\mathbf{E}_{*}}$, $s_{E_{u,-}} \vert_{\mathbf{E}_{*}^{1}} = s_{\mathbf{E}_{*}}$, and
\begin{align*}
	s_{E_{u,-}}(d_1) &= u	& r_{E_{u,-}}(d_1) &= v_{1} \\
	s_{E_{u,-}}(d_2) &= v_1	& r_{E_{u,-}}(d_2) &= u.
\end{align*}
Moreover, $E_{u,--}$ can be described as follows (up to canonical isomorphism):
\begin{align*}
E_{u,--}^{0} &= E^{0} \sqcup \mathbf{E}_{**}^{0} \\
E_{u,--}^{1} &= E^{1} \sqcup \mathbf{E}_{**}^{1} \sqcup \{ d_1, d_2 \}
\end{align*}
with $r_{E_{u,--}} \vert_{E^{1}} = r_{E}$, $s_{E_{u,--}} \vert_{ E^{1} } = s_{E}$, $r_{E_{u,--}} \vert_{\mathbf{E}_{**}^{1}} = r_{\mathbf{E}_{**}}$, $s_{E_{u,--}} \vert_{\mathbf{E}_{**}^{1}} = s_{\mathbf{E}_{**}}$, and
\begin{align*}
	s_{E_{u,--}}(d_1) &= u		& r_{E_{u,--}}(d_1) &= w_{1} \\
	s_{E_{u,--}}(d_2) &= w_1	& r_{E_{u,--}}(d_2) &= u.
\end{align*}
\end{notation}

\begin{proposition} \label{prop:cuntzsplicetwice}
Let $E$ be a finite graph with no sinks and no sources satisfying Condition~(K) such that $\Bsf_E\in\MPplusZ[\mathbf{n}]$. 
Let $u$ be a vertex in $E^{0}$, and assume that it belongs to the component corresponding to the block $j\in\calP$.  
Then $\Bsf_{E_{u,--}}$ is an element of $\MPZ[\mathbf{n}+4\mathbf{e}_j]$, and there are an \SLPEe $(U,V)$ from $-\iota_{4\mathbf{e}_j}(-\Bsf_E)$ to  $\Bsf_{E_{u,--}}$ and a \calP-equivariant isomorphism $\Phi$ from $C^*(E)\otimes\K$ to $C^*(E_{u,--} )\otimes\K$ such that $\FKR(\calP;\Phi)=\FKRs(U,V)$. 
\end{proposition}
\begin{proof}
The $j$'th diagonal block of $\Bsf_{E_{u,--}}$ can be described as the matrix 
$$\begin{pmatrix}
\Bsf_{E}\{j\} & \begin{smallpmatrix} 0&0&0&0 \\ \vdots & \vdots & \vdots & \vdots \\ 1 & 0 & 0 & 0 \end{smallpmatrix} \\ 
\begin{smallpmatrix}0 & \cdots & 1 \\ 0 & \cdots & 0 \\ 0 & \cdots & 0 \\ 0 & \cdots & 0  \end{smallpmatrix} & \begin{smallpmatrix}0 & 1 & 1 & 0 \\ 1 & 0 & 0 & 0 \\ 1 & 0 & 0 & 1 \\ 0 & 0 & 1 & 0\end{smallpmatrix}
\end{pmatrix}.$$
Let $U,V\in\SLPZ[\mathbf{n}+4\mathbf{e}_j]$ be the identity matrices everywhere except for the $j$'th diagonal block where they are 
$$\begin{pmatrix}
I& \begin{smallpmatrix} 0&0&0&0 \\ \vdots & \vdots & \vdots & \vdots \\ 0 & 1 & 0 & 0 \end{smallpmatrix} \\ 
\begin{smallpmatrix}0 & \cdots & 0 \\ 0 & \cdots & 0 \\ 0 & \cdots & 0 \\ 0 & \cdots & 0  \end{smallpmatrix} & \begin{smallpmatrix}1 & 0 & 0 & 1 \\ 0 & 1 & 0 & 0 \\ 0 & 0 & 1 & 0 \\ 0 & 0 & 0 & 1\end{smallpmatrix}
\end{pmatrix}$$
$$\begin{pmatrix}
I& \begin{smallpmatrix} 0&0&0&0 \\ \vdots & \vdots & \vdots & \vdots \\ 0 & 0 & 0 & 0 \end{smallpmatrix} \\ 
\begin{smallpmatrix}0 & \cdots & -1 \\ 0 & \cdots & 0 \\ 0 & \cdots & 0 \\ 0 & \cdots & 0  \end{smallpmatrix} & \begin{smallpmatrix}0 & -1 & 0 & 0 \\ -1 & 0 & 0 & 0 \\ -1 & 0 & 0 & -1 \\ 0 & 0 & -1 & 0\end{smallpmatrix}
\end{pmatrix} ,$$
respectively. It is easy to verify that these blocks have indeed determinant $1$ and that $(U,V)$ is an \SLPEe from $-\iota_{4\mathbf{e}_j}(-\Bsf_E)$ to $\Bsf_{E_{u,--}}$. 
The idea now is to deduce the existence of $\Phi$ from \cite[Theorem 4.4]{MR1907894} and Proposition~\ref{prop:toke}, but $E$ and $E_{u,--}$ do not satisfy the assumption of \cite[Theorem 4.4]{MR1907894}. We therefore make a small detour. 
By using column and row additions in the same way as in the proof of Lemma~\ref{lem:enlargen} together with Proposition~\ref{prop:toke} we see that we get two finite graphs $F$ and $F'$ satisfying Condition~(K) with no sinks and no sources such that $\Bsf_F,\Bsf_{F'}\in\MPplusZ[\mathbf{n}+4\mathbf{e}_j]$ and such that there exist \calP-equivariant isomorphisms \fctw{\Psi}{C^{*}(E)\otimes\K}{C^{*}(F)\otimes\K} and \fctw{\Psi'}{C^{*}(E_{u,--})\otimes\K}{C^{*}(F')\otimes\K} with $\FKR(\calP;\Psi)=\FKRs(U_1,V_1)$ and $\FKR(\calP;\Psi')=\FKRs(U_1',V_1')$, where $(U_1,V_1)$ and $(U_1',V_1')$ are \SLPEe{s} from $-\iota_{4\mathbf{e}_j}(-\Bsf_E)$ to $\Bsf_F$ and from $\Bsf_{E_{u,--}}$ to $\Bsf_{F'}$, respectively. The composition thus gives an \SLPEe $(U',V')$ from $\Bsf_F$ to $\Bsf_{F'}$. 

It now follows from \cite[Theorem 4.4]{MR1907894} and Proposition~\ref{prop:toke} that there is a \calP-equivariant isomorphism \fctw{\Phi'}{C^*(F)\otimes\K}{C^*(F')\otimes\K} such that $\FKR(\calP,\Phi')=\FKRs(U',V')$. Then $\Phi:=(\Psi')^{-1}\circ\Phi'\circ\Psi$ is a \calP-equivariant isomorphism from $C^*(E)\otimes\K$ to $C^*(E_{u,--})\otimes\K$ such that $\FKR(\calP,\Phi)=\FKRs(U,V)$.
\end{proof}

The following theorem is the key result to prove that the Cuntz splice induces stably isomorphic graph \cas in a way where we can control the induced maps on the reduced filtered $K$-theory. 
Since it does not increase the difficulty of the proof, for the convenience of possible later application, we state and prove the result in slightly greater generality than needed in this paper --- allowing for infinite emitters and sinks. 

\begin{theorem}\label{thm:cuntz-splice-1}
Let $E$ be a graph with finitely many vertices, and let $u$ be a regular vertex in $E^0$.  Assume that $\Bsf_E \in \MPZccc$ and that $u$ belongs to the block $j\in\calP$.  Then there exist $U\in\SLPZ[\mathbf{m}+4\mathbf{e}_j]$ and $V\in\GLPZ[\mathbf{n}+4\mathbf{e}_j]$ that are the identity matrix everywhere except for the $j$'th diagonal block, all the diagonal blocks of $U$ and $V$ have determinant $1$ except for the $j$'th diagonal block of $V$, which has determinant $-1$, and $(U,V)$ is a \GLPEe from $-\iota_{2\mathbf{e}_j}(-\Bsf_{E_{u,-}}^\bullet)$ to  $\Bsf_{E_{u,--}}^\bullet$. 
Moreover, there exists a $\calP$-equivariant isomorphism $\Phi$ from $C^{*}(E_{u,-})$ to $C^{*}(E_{u,--})$ such that $\FKR(\calP;\Phi)=\FKR(U,V)$. 
\end{theorem}

\begin{proof}
In the proof of \cite[Theorem~\ref{CS-t:cuntz-splice-1}]{arXiv:1602.03709v2} we found two representations of $C^*(E_{u,-})$ and $C^*(E_{u,--})$ inside a \ca $\A$, and showed that they are in fact equal. 
Now, we will produce $U$ and $V$ such that the identity isomorphism is exactly given by the isomorphism $\FKR(U,V)$ --- when we use the canonical identifications of the $K$-theory. We refer the reader to the detailed definitions in the proof of \cite[Theorem~\ref{CS-t:cuntz-splice-1}]{arXiv:1602.03709v2}. 

Let $U\in\SLPZ[\mathbf{m}+4\mathbf{e}_j]$ and $V\in\GLPZ[\mathbf{n}+4\mathbf{e}_j]$ be the identity matrix everywhere except for the $j$'th diagonal block where they are 
\begin{align*}
U_0&=\begin{pmatrix}
I & 0 \\ 
0 & \begin{smallpmatrix}1 & 0 & 0 & 1 \\ 0 & 1 & 0 & 0 \\ 0 & 0 & 1 & 0 \\ 0 & 0 & 0 & 1\end{smallpmatrix}
\end{pmatrix}, \\ 
V_0 &= \begin{pmatrix}
I& 0 \\ 
0 & \begin{smallpmatrix}1 & 0 & 0 & 0 \\ 0 & 1 & 0 & 0 \\ -1 & 0 & 0 & -1 \\ 0 & 0 & -1 & 0\end{smallpmatrix}
\end{pmatrix},
\end{align*}
respectively. It is clear that all the diagonal blocks of $U$ and $V$ have determinant $1$ except for the $j$'th diagonal block of $V$, which has determinant $-1$. Moreover, it is easy to see that $(U,V)$ is a \GLPEe from $-\iota_{2\mathbf{e}_j}(-\Bsf_{E_{u,-}}^\bullet)$ to  $\Bsf_{E_{u,--}}^\bullet$ (where we let the order of the added vertices follow the orders $v_1,v_2$ and $w_1,w_2,w_3,w_4$, respectively). 
We will show that the identity isomorphism is exactly given by the isomorphism $\FKR(U,V)$ --- when we use the canonical identifications of the $K$-theory. 

Recall from the proof of \citefirst[Theorem~\ref{CS-t:cuntz-splice-1}] that $\{q_v,t_e\}$ is the Cuntz-Krieger $E_{u,-}$-family in $\A$ that was constructed and $\{P_v,S_e\}$ is the Cuntz-Krieger $E_{u,--}$-family in $\A$ that was constructed, and that we call these two families $\mathcal{S}$ and $\mathcal{T}$, respectively.  Moreover, since $E$ has finitely many vertices, these \cas are all unital sub-\cas.  

It is clear that we have that $[q_{v_1}]_0=[P_{w_1}]_0=[P_{w_3}]_0=[P_{w_4}]_0=0$ and $[q_{v_2}]_0=-[q_u]_0=-[P_u]_0=[P_{w_2}]_0$ in $K_0(\mathfrak{I})$ for any gauge invariant ideal $\mathfrak{I}$ containing $q_{v_i}$, $P_{w_j}$, $i=1,2$, $j=1,2,3,4$, and similarly for subquotients. 
Since we have $P_v=q_v$ for all $v\in E^0$, it now follows that this \GLPEe induces exactly the desired maps on the $K_0$-groups. 

Now we need to show that on the gauge simple subquotients, we get the right identification of $K_1$. For all subquotients except for the $j$'th, this is clear. 
For notational convenience, assume that $u$ corresponds to the last row of the diagonal block $\Bsf_E^\bullet\{j\}$, \ie, the $m_j$'th row. 
We may also without loss of generality assume that $j$ is the only component of $E$. 
Now let $\mathbf{x}=(x_i)_{i=1}^{m_j+4}\in\ker \left(-\iota_{2\mathbf{e}_j}\left(-\Bsf_{E_{u,-}}^\bullet\right)\right)^\mathsf{T}$. 
Note that $x_{m_j+1}=x_{m_j+3}=x_{m_j+4}=0$ and $x_{m_j}+x_{m_j+2}=0$. 
Note also that this implies that $\mathbf{x}\in\ker \left(\Bsf_{E_{u,--}}^\bullet\right)^\mathsf{T}$.
Now we construct the partial isometries $\mathsf{S}_\mathbf{x}$ and $\tilde{\mathsf{S}}_\mathbf{x}$, the projections $\mathsf{P}_\mathbf{x}=\mathsf{S}_\mathbf{x}^*\mathsf{S}_\mathbf{x}
=\mathsf{S}_\mathbf{x}\mathsf{S}_\mathbf{x}^*$,
$\tilde{\mathsf{P}}_\mathbf{x}
=\tilde{\mathsf{S}}_\mathbf{x}^*\tilde{\mathsf{S}}_\mathbf{x}
=\tilde{\mathsf{S}}_\mathbf{x}\tilde{\mathsf{S}}_\mathbf{x}^*$ and the unitaries 
$\mathsf{U}_\mathbf{x}=\mathsf{S}_\mathbf{x}+1-\mathsf{P}_\mathbf{x}$ and 
$\tilde{\mathsf{U}_\mathbf{x}}
=\tilde{\mathsf{S}}_\mathbf{x}+1-\tilde{\mathsf{P}}_\mathbf{x}$ in 
matrix algebras over $C^*(\mathcal{S})=C^*(\mathcal{T})$ as in \cite{MR2922394}, where the one without a tilde correspond to the family $\mathcal{S}$ while the ones with a tilde correspond to the family $\mathcal{T}$. 
From the matrix $(U_0^\mathsf{T})^{-1}$ we see that what we need to show is that 
$[\mathsf{U}_\mathbf{x}]_1=[\tilde{\mathsf{U}}_\mathbf{x}]_1$. 

In \cite{MR2922394} we have the following definitions:
\begin{align*}
L_v^+&=\setof{(e,i)}{e\in E_{u,-}^1,r(e)=v, 1\leq i\leq -x_{s(e)}}\cup\setof{(v,i)}{1\leq i\leq x_v}, \\
L_v^-&=\setof{(e,i)}{e\in E_{u,-}^1,r(e)=v, 1\leq i\leq x_{s(e)}}\cup\setof{(v,i)}{1\leq i\leq -x_v}, \\ 
\tilde{L}_w^+&=\setof{(f,k)}{e\in E_{u,--}^1,r(e)=w, 1\leq k\leq -x_{s(f)}}\cup\setof{(w,k)}{1\leq k\leq x_w}, \\
\tilde{L}_w^-&=\setof{(f,k)}{e\in E_{u,--}^1,r(e)=w, 1\leq k\leq x_{s(e)}}\cup\setof{(w,k)}{1\leq k\leq -x_w},
\end{align*}
for $v\in E_{u,-}^0$ and all $w\in E_{u,--}^0$. 
Since $x_{m_j+1}=x_{m_j+3}=x_{m_j+4}=0$, neither any of the edges $d_2$, $e_1$, $e_2$, $f_1$, $f_2$, $f_5$, $f_6$, $f_7$, $f_8$, $f_9$, $f_{10}$ nor any of the vertices $v_1$, $w_1$, $w_3$, $w_4$ will appear in any such pair. 
It is clear that $L_v^+=\tilde{L}_v^+$ and $L_v^-=\tilde{L}_v^-$ for all $v\in E^0$.
Also we see that $L_{w_i}^+=\emptyset=\tilde{L}_{w_i}^-$ for $i=3,4$. 
The substitutions $v_2\mapsto w_2$, $e_3\mapsto f_3$ and $e_4\mapsto f_4$ will define canonical bijections from $L_{v_i}^+$ to $\tilde{L}_{w_i}^+$ and from $L_{v_i}^-$ to $\tilde{L}_{w_i}^-$, for $i=1,2$. 

Thus the size of $\mathsf{U}_\mathbf{x}$ and $\tilde{\mathsf{U}}_\mathbf{x}$ is the same. 
We want to show that $[\mathsf{U}_\mathbf{x}]_1=[\tilde{\mathsf{U}}_\mathbf{x}]_1$.  Since $[\mathsf{U}_{-\mathbf{x}}]_1=-[\mathsf{U}_{\mathbf{x}}]_1$ and $[\tilde{\mathsf{U}}_{-\mathbf{x}}]_1=-[\tilde{\mathsf{U}}_{\mathbf{x}}]_1$, we may assume that $x_u=x_{m_j}\geq 0$.
With a slight abuse of notation, we let $z$ denote the diagonal matrix of this matrix algebra having $z$ as all diagonal entries. 
Using the definitions and matrix units from \cite{MR2922394}, we see that
\begin{align*}
z\mathsf{U}_\mathbf{x}z^*-\tilde{\mathsf{U}}_\mathbf{x}
&=\sum_{1\leq i\leq x_u}\Big(
\left(zt_{e_3}^*z^*-s_{f_3}^*\right)\mathsf{E}_{[e_3,i],\langle v_2,i\rangle} 
+\left(zt_{e_4}^*z^*-s_{f_4}^*\right)\mathsf{E}_{[e_4,i],\langle v_2,i\rangle} \\
&\qquad +\left(P_{w_2}-zq_{v_2}z^*\right)\mathsf{E}_{\langle v_2,i\rangle,\langle v_2,i\rangle}
\Big).
\end{align*}
Now we see that
\begin{align*}
z\mathsf{U}_\mathbf{x}z^*\tilde{\mathsf{U}}_\mathbf{x}^*
&=1+(z\mathsf{U}_\mathbf{x}z^*-\tilde{\mathsf{U}}_\mathbf{x})\tilde{\mathsf{U}}_\mathbf{x}^* \\ 
&=1+\sum_{1\leq i\leq x_u}\bigg(
\sum_{j=3}^4\Big(
\left(zt_{e_j}^*z^*-zt_{e_j}^*z^*P_{w_2}\right)\mathsf{E}_{[e_j,i],\langle v_2,i,\rangle}
+\sum_{k=3}^4\left(zt_{e_j}^*z^*-s_{e_j}^*\right)s_{e_k}\mathsf{E}_{[e_j,i],[e_k,i]}
\Big)\bigg).
\end{align*}
Note that this is in fact a unitary in $M_h(\A)$ of the form $1_{M_h(\A_0)}+U$ where $U$ is the unitary 
\begin{align*}
U&=1_{M_h(\mathcal{E})}+\sum_{1\leq i\leq x_u}\bigg(
\sum_{j=3}^4\Big(
\left(zt_{e_j}^*z^*-zt_{e_j}^*z^*P_{w_2}\right)\mathsf{E}_{[e_j,i],\langle v_2,i,\rangle}
+\sum_{k=3}^4\left(zt_{e_j}^*z^*-s_{e_j}^*\right)s_{e_k}\mathsf{E}_{[e_j,i],[e_k,i]}
\Big)\bigg)
\end{align*}
in $M_h(\mathcal{E})$ and $h$ is the number of elements in $\mathsf{L}_\mathbf{x}^+$. 
Since $K_1(\mathcal{E})=0$ it follows from the split exactness of $K_1$ that 
$[z\mathsf{U}_\mathbf{x}z^*\tilde{\mathsf{U}}_\mathbf{x}^*]_1=0$ in $K_1(\A)$. 
Since the embedding of \A into $C^*(\mathcal{S})=C^*(\mathcal{T})$ is unital, it follows that $[z\mathsf{U}_\mathbf{x}z^*\tilde{\mathsf{U}}_\mathbf{x}^*]_1=0$ in $K_1(C^*(\mathcal{S}))$. 
Thus $[\mathsf{U}_\mathbf{x}]_1=[\tilde{\mathsf{U}}_\mathbf{x}]_1$ in $K_1(C^*(\mathcal{S}))$. 
\end{proof}

By combining the above, we get the following important result. 

\begin{corollary}\label{cor:cuntzspliceinvariant}
Let $E$ be a finite graph with no sources and no sinks satisfying Condition~(K) 
such that $\Bsf_E \in \MPplusZ[\mathbf{n}]$. Let $u$ be a vertex in $E^0$ and assume that $u$ belongs to the block $j\in\calP$.  
Then $\Bsf_{E_{u,-}}\in\MPZccc[\mathbf{n}+2\mathbf{e}_j]$, and there exist $U\in\SLPZ[\mathbf{n}+2\mathbf{e}_j]$ and $V\in\GLPZ[\mathbf{n}+2\mathbf{e}_j]$ such that the following holds. The matrices $U$ and $V$ are equal to the identity matrices everywhere except for the $j$'th diagonal block. All the diagonal blocks of $U$ and $V$ have determinant $1$ except for the $j$'th diagonal block of $V$, which has determinant $-1$ and $(U,V)$ is a \GLPEe from $-\iota_{2\mathbf{e}_j}(-\Bsf_{E})$ to  $\Bsf_{E_{u,-}}$. 
Moreover, there exists a $\calP$-equivariant isomorphism $\Phi$ from $C^{*}(E) \otimes \K$ to $C^{*}(E_{u,-}) \otimes \K$ such that $\FKR(\calP;\Phi)=\FKRs(U,V)$. 
\end{corollary}
\begin{proof}
It follows from Theorem~\ref{thm:cuntz-splice-1} and Proposition~\ref{prop:cuntzsplicetwice} that the statement holds if we say that the \calP-equivariant isomorphism $\Phi$ is induced by such a \GLPEe from $-\iota_{4\mathbf{e}_j}(-\Bsf_{E})$ to  $-\iota_{2\mathbf{e}_j}(-\Bsf_{E_{u,-}})$. 
Let $(U_1,V_1)$ be the \SLPEe from the proof of Proposition~\ref{prop:cuntzsplicetwice} and let $(U_2,V_2)$ be the \GLPEe from the proof of Theorem~\ref{thm:cuntz-splice-1}. 
By computing $U_2^{-1}U_1$ and $V_1V_2^{-1}$ we see that these are in fact of the form $\iota_{2\mathbf{e}_j}(U)$ and $\iota_{2\mathbf{e}_j}(V)$, respectively, for matrices $U\in\SLPZ[\mathbf{n}+2\mathbf{e}_j]$ and $V\in\GLPZ[\mathbf{n}+2\mathbf{e}_j]$ of the desired form. 
\end{proof}

\begin{theorem}\label{thm:X-stable-strong}
Let $A$ and $A'$ be nondegenerate $\{0,1\}$-matrices satisfying Condition~(II). 
Let $X=\Prim(\Ocal_A)$ and assume that there is a homeomorphism between $\Prim(\Ocal_A)$ and $\Prim(\Ocal_{A'})$, and use this homeomorphism to view $\Ocal_A\otimes\K$ and $\Ocal_{A'}\otimes\K$ as $X$-algebras. 
Assume that $\fct{\varphi}{\FKR(X;\Ocal_A\otimes\K)}{\FKR(X;\Ocal_{A'}\otimes\K)}$ is an isomorphism. 
Then there exists an $X$-equivariant isomorphism $\fct{\Phi}{\Ocal_A\otimes\K}{\Ocal_{A'}\otimes\K}$ such that $\FKR(X;\Phi)=\varphi$. 
\end{theorem}

\begin{proof}
\textbf{Step 1: }
According to Lemma~\ref{lem:stdform}, we may consider two finite graphs $E$ and $F$ with no sinks and no sources satisfying Condition~(K) such that $\Bsf_E\in\MPplusZ[\mathbf{n}_E]$ and $\Bsf_F\in\MPplusZ[\mathbf{n}_F]$ and then prove the corresponding result for the graph \cas $C^{*}(E)$ and $C^{*}(F)$. 
We choose the block structure in such a way that it is over the same poset \calP, and such that the block structures are compatible with the homeomorphism between $\Prim(\Ocal_A)$ and $\Prim(\Ocal_{A'})$, \cf\ the discussing right after \citefirst[Definition~\ref{I-def:circ}].

\textbf{Step 2: }
According to Lemma~\ref{lem:enlargen}, we may, moreover, assume that $\Asf_E$ and $\Asf_F$ have the same block structure. So we write $\Bsf_E,\Bsf_F\in\MPplusZ[\mathbf{n}]$, for some multiindex $\mathbf{n}$. 

\textbf{Step 3: }
It follows from \citefirst[Section~\ref{I-sec:red-filtered-K-theory-K-web-GLP-and-SLP-equivalences}] and Corollary~\ref{cor:BH-4.6-B} that there exists a \GLPEe \fctw{(U,V)}{\Bsf_E}{\Bsf_F} such that $\varphi=\FKRs(U,V)$. 

\textbf{Step 4: }
For each $i\in\calP$ with $\det(U\{i\})\neq\det(V\{i\})$, we do as follows. 
Assume that $j\in\calP$ with $\det(U\{j\})\neq\det(V\{j\})$. 
Let $v$ be a vertex in the component corresponding to $j$. 
Cuntz-splice the graph $E$ at vertex $v$, and call the new graph $E_1$. Note that 
$\Bsf_{E_1}\in\MPZ[\mathbf{n}+2\mathbf{e}_j]$. 
Corollary~\ref{cor:cuntzspliceinvariant} gives us an $X$-equivariant isomorphism  \fctw{\Phi_1}{C^{*}(E_1)\otimes\K}{C^{*}(E)\otimes\K} and a \GLPEe \fctw{(U_1,V_1)}{\Bsf_{E_1}}{-\iota_{2\mathbf{e}_j}(-\Bsf_{E})} such that $\FKR(X;\Phi_1)=\FKRs(U_1,V_1)$, $\det(U_1\{j\})\neq\det(V_1\{j\})$ and $\det(U_1\{j\})=\det(V_1\{j\})=1$, for all $i\neq j$. 
By composing the \GLPEe{s} $(U_1,V_1)$ and $(U,V)$ and using column and row additions together with Proposition~\ref{prop:toke} in the same way as in the proof of Lemma~\ref{lem:enlargen}, we see that we may without loss of generality assume that $\det(U\{i\})=\det(V\{i\})$, for all $i\in\calP$.

\textbf{Step 5: }
Consequently, in addition to the assumptions on $E$ and $F$, we now assume that $\det(U\{i\})=\det(V\{i\})$, for all $i\in\calP$. 
Let $\mathbf{r}$ be the multiindex where all entries are $2$. 
Now let $U',V'\in\SLPZ[\mathbf{n}+\mathbf{r}]$ be given as follows. 
We obtain $U'$ from $\iota_\mathbf{r}(U)$ by interchanging the last two rows of each row-block $i$ that satisfies $\det(U\{i\})=-1$, and we similarly obtain $V'$ from $\iota_\mathbf{r}(V)$ by interchanging the last two rows of each row-block $i$ that satisfies $\det(V\{i\})=-1$. 
It is now clear that $(U',V')$ is an \SLPEe from $-\iota_\mathbf{r}(-\Bsf_E)$ to $-\iota_\mathbf{r}(-\Bsf_F)$ and it is also clear that $\FKRs(U,V)=\FKRs(U',V')$. 
By using column and row additions in the same way as in the proof of Lemma~\ref{lem:enlargen} together with Proposition~\ref{prop:toke} we see that we get two finite graphs $E'$ and $F'$ satisfying Condition~(K) with no sinks and no sources such that $\Bsf_{E'},\Bsf_{F'}\in\MPplusZ[\mathbf{n}+\mathbf{r}]$ and such that there exist \calP-equivariant isomorphisms \fctw{\Psi}{C^{*}(E)\otimes\K}{C^{*}(E')\otimes\K} and \fctw{\Psi'}{C^{*}(F)\otimes\K}{C^{*}(F')\otimes\K} with $\FKR(\calP;\Psi)=\FKRs(U_1,V_1)$ and $\FKR(\calP;\Psi')=\FKRs(U_2,V_2)$, where $(U_1,V_1)$ and $(U_2,V_2)$ are \SLPEe{s} from $-\iota_{\mathbf{r}}(-\Bsf_E)$ to $\Bsf_{E'}$ and from $-\iota_{\mathbf{r}}(-\Bsf_{F})$ to $\Bsf_{F'}$, respectively. Thus by composition, the problem is reduced to lifting the filtered $K$-theory isomorphism $\FKRs(\tilde{U},\tilde{V})$ induced by an \SLPEe $(\tilde{U},\tilde{V})=(U_2U'U_1^{-1},V_1^{-1}V'V_2)$ from $\Bsf_F$ to $\Bsf_{F'}$. 

\textbf{Step 6: }
Now the theorem follows from Proposition~\ref{prop:toke} and~\cite[Theorem~4.4]{MR1907894}.
\end{proof}

We now show how Theorem~\ref{thm:X-stable-strong} can be used to prove Theorem~\ref{thm:strong-stable}.

\begin{proof}[Proof of Theorem~\ref{thm:strong-stable}]
Let $A$ and $A'$ be nondegenerate matrices with entries from $\{0,1\}$ satisfying Condition~(II).  Suppose there exists an isomorphism \fctw{(\kappa,\rho)}{\FK_{\mathrm{red}}(\Ocal_A)}{\FK_{\mathrm{red}}(\Ocal_{A'})}.  Since $\rho$ is a homeomorphism from $\Prim(\Ocal_A)$ to $\Prim(\Ocal_{A'})$, we get that $(\Ocal_{A'},\rho^{-1})$ is a \ca over $\Prim(\Ocal_A)$.  Now, $\kappa$ becomes an isomorphism from $\FKR(\Prim(\Ocal_A);\Ocal_A)$ to $\FKR(\Prim(\Ocal_A);\Ocal_{A'})$.  By Theorem~\ref{thm:X-stable-strong}, there exists a $\Prim(\Ocal_A)$-equivariant isomorphism $\fct{\Psi}{\Ocal_A\otimes\K}{\Ocal_{A'}\otimes\K}$ such that $\FKR(\Prim(\Ocal_A);\Psi)=\kappa$ up to canonical identification.  By construction, $\Psi$ induces $\rho$.  Hence, $\fct{\Psi}{\Ocal_A\otimes\K}{\Ocal_{A'}\otimes\K}$ is an isomorphism such that $\FK^s_{\mathrm{red}}(\Psi)=(\kappa,\rho)$. 
\end{proof}

\begin{theorem}\label{thm:X-unital-strong}
Let $A$ and $A'$ be nondegenerate $\{0,1\}$-matrices satisfying Condition~(II). 
Let $X=\Prim(\Ocal_A)$ and assume that there is a homeomorphism between $\Prim(\Ocal_A)$ and $\Prim(\Ocal_{A'})$, and use this homeomorphism to view $\Ocal_A$ and $\Ocal_{A'}$ as $X$-algebras. 
Assume that $\fct{\varphi}{\FKR(X;\Ocal_A)}{\FKR(X;\Ocal_{A'})}$ is an isomorphism that preserves the class of the unit in $K_0$. 
Then there exists an $X$-equivariant isomorphism $\fct{\Phi}{\Ocal_A}{\Ocal_{A'}}$ such that $\FKR(X;\Phi)=\varphi$. 
\end{theorem}

\begin{proof}
This follows from Definition~\ref{def:red-filtered-Ktheory-class-unit}, Theorem~\ref{thm:X-stable-strong}, \cite[Lemma~8.3 and Definition~8.4]{MR3349327}, and \cite[Theorem~3.3]{arXiv:1301.7695v1} (by \cite[Corollary 7.2]{MR2310414}, row-finite graph \cas have stable weak cancellation property). 
\end{proof}

\noindent \emph{Proof of Theorem~\ref{thm:strong-unital}.}  This can be proved in a similar fashion to how Theorem~\ref{thm:strong-stable} was proved, using Theorem~\ref{thm:X-unital-strong} instead of Theorem~\ref{thm:X-stable-strong}.\qed

\section*{Acknowledgements}

This work was partially supported by a grant from the Simons Foundation (\# 279369 to Efren Ruiz).  Some of the work was conducted while
all three authors were attending the research program
\emph{Classification of operator algebras: complexity, rigidity, and                                                              
 dynamics} at the Mittag-Leffler Institute, January--April 2016. We
thank the institute and its staff for the excellent work conditions
provided. The second named author thanks University of Hawaii, Hilo for their hospitality during a visit, where part of this work was done.




\providecommand{\bysame}{\leavevmode\hbox to3em{\hrulefill}\thinspace}
\providecommand{\MR}{\relax\ifhmode\unskip\space\fi MR }
\providecommand{\MRhref}[2]{%
  \href{http://www.ams.org/mathscinet-getitem?mr=#1}{#2}
}
\providecommand{\href}[2]{#2}

\end{document}